\documentclass[10pt]{amsart}
\usepackage{amsfonts}
\usepackage{mathrsfs}
\usepackage{latexsym}
\usepackage{graphicx}
\usepackage{amscd,amssymb,amsmath,amsbsy,amsthm}
\usepackage[all]{xy}
\usepackage[colorlinks,plainpages,urlcolor=black,citecolor=black,linkcolor=black]{hyperref}
\usepackage{fancyhdr}

\usepackage[backend=bibtex]{biblatex}
\bibliography{Masseybib.bib}

\newcommand{\Span}{{\operatorname{Span}}}

\newcommand{\Diff}{{\operatorname{Diff}}}

\DeclareMathOperator{\rank}{rank}
\DeclareMathOperator{\Id}{id}

\DeclareMathOperator{\Top}{top}

\newcommand{\U}{{\mathcal U}}
\renewcommand{\t}{{\mathbf t}}
\newcommand{\0}{{\mathbf 0}}
\newcommand{\C}{{\mathbb C}}
\newcommand{\PP}{{\mathbb P}}
\newcommand{\z}{{\mathbf z}}
\newcommand{\Z}{{\mathbb Z}}

\newcommand{\Q}{{\mathbb Q}}

\newcommand{\D}{{\mathbb D}}

\newcommand{\cO}{{\mathcal O}}

\newcommand{\hyp}{{\mathbb H}}
\newcommand{\supp}{\operatorname{supp}}

\newcommand{\im}{\mathop{\rm im}\nolimits}

\newcommand{\Adot}{\mathbf A^\bullet}
\newcommand{\Bdot}{\mathbf B^\bullet}
\newcommand{\Cdot}{\mathbf C^\bullet}
\newcommand{\Idot}{\mathbf I^\bullet}
\newcommand{\Fdot}{\mathbf F^\bullet}

\newcommand{\Pdot}{\mathbf P^\bullet}

\newcommand{\Ndot}{\mathbf  N^\bullet}

\newcommand{\pd}{\partial}

\newcommand{\wt}{\widetilde}

\usepackage [english]{babel}
\newtheorem{defn0}{Definition}[subsection]
\newtheorem{prop0}[defn0]{Proposition}
\newtheorem{conj0}[defn0]{Conjecture}
\newtheorem{thm0}[defn0]{Theorem}
\newtheorem{lem0}[defn0]{Lemma}
\newtheorem{corollary0}[defn0]{Corollary}
\newtheorem{example0}[defn0]{Example}
\newtheorem{remark0}[defn0]{Remark}
\newtheorem{question0}[defn0]{Question}

\newenvironment{defin}{\begin{defn0}}{\end{defn0}}
\newenvironment{prop}{\begin{prop0}}{\end{prop0}}
\newenvironment{conj}{\begin{conj0}}{\end{conj0}}
\newenvironment{thm}{\begin{thm0}}{\end{thm0}}
\newenvironment{lem}{\begin{lem0}}{\end{lem0}}
\newenvironment{cor}{\begin{corollary0}}{\end{corollary0}}
\newenvironment{exm}{\begin{example0}\rm}{\end{example0}}
\newenvironment{rem}{\begin{remark0}\rm}{\end{remark0}}

\newcommand{\defref}[1]{Definition~\ref{#1}}
\newcommand{\propref}[1]{Proposition~\ref{#1}}
\newcommand{\thmref}[1]{Theorem~\ref{#1}}
\newcommand{\lemref}[1]{Lemma~\ref{#1}}
\newcommand{\corref}[1]{Corollary~\ref{#1}}
\newcommand{\exref}[1]{Example~\ref{#1}}
\newcommand{\secref}[1]{Section~\ref{#1}}

\newcommand{\remref}[1]{Remark~\ref{#1}}

\newcommand{\equref}[1]{Formula~\ref{#1}}

\newcommand{\mbf}[1]{{\mathbf #1}}
\begin{document}

\title{Deformation Formulas for Parameterized Hypersurfaces}

\author{Brian Hepler}

\maketitle

\begin{abstract}
 We investigate one-parameter deformations of functions on affine space which define parameterizable hypersurfaces. With the assumption of isolated polar activity at the origin, we are able to completely express the L\^{e} numbers of the special fiber in terms of the L\^{e} numbers of the generic fiber and the characteristic polar multiplicities of the comparison, a perverse sheaf naturally associated to any reduced complex analytic space on which the constant sheaf $\Q_X^\bullet[\dim X]$ is perverse. This generalizes the classical formula for the Milnor number of a plane curve in terms of double points as well as Mond's image Milnor number. We also recover results of Gaffney and Bobadilla using this framework. We obtain similar deformation formulas for maps from $\C^2$ to $\C^3$, and provide an ansatz for obtaining deformation formulas for all dimensions within Mather's nice dimensions.
 \end{abstract}

\section{Generalizing Milnor's Formula to Higher Dimensions}\label{sec:deformationformulas}

Suppose that $\U$ is an open neighborhood of the origin in $\C^2$. Let $f_0:(\U, \0)\rightarrow (\C, 0)$ be a complex analytic function which has an isolated critical point at the origin. Thus, $f_0$ defines a plane curve $V(f_0)$ in $\U$. Let $r$ be the number of irreducible components of $V(f_0)$ at the origin. Then, by a well-known result of Milnor (Theorem 10.5 of \cite{milnorsing}), the Milnor number $\mu_\0(f_0)$ is related to the number of double points $\delta$ which occur in a generic (stable) deformation of $f_0$ by 
\begin{equation}\label{eqn:milnor}
\mu_\0(f_0) = 2\delta -r +1. 
\end{equation}
We wish to generalize this formula, in light of recent work with the author and David Massey in \cite{hepmasparam} (Theorem 5.3), in which we obtain a quick proof of the above formula.  

 In re-proving Milnor's formula (\ref{eqn:milnor}) in \cite{hepmasparam}, one immediately notices that the generality of the methods used in \cite{hepmasparam} are not at all limited to deformations of curves in $\C^2$; consequently, it is natural to hope that a similar, more general result holds between the vanishing cycles and the perverse sheaf $\Ndot_{V(f)}$ (central to this current paper and \cite{hepmasparam}) in deformations of parameterized hypersurfaces. We prove such a generalization in this paper, and obtain a similar formula for deformations of parameterized surfaces in $\C^3$, and a ``bootstrap ansatz" for obtaining such results for deformations of parameterized hypersurfaces in $\C^{n+1}$ if one knows all of the stable maps from $\C^{n+1}$ to $\C^{n+2}$. This generalizes work of David Mond's image Milnor number \cite{mondbentwires}, similar deformation formulas of Massey and Dirk Siersma \cite{masseysiersma}, work of Terence Gaffney \cite{gaffneypairs}, \cite{l0equivalence}, in addition to Milnor's original formula. We also recover a result of Javier Fern\'{a}ndez de Bobadilla regarding a special case of L\^{e}'s Conjecture regarding the equisingularity of surfaces in $\C^3$ with smooth normalization \cite{bobleconj}.

\medskip

 The first question we ask is: \textbf{what if we didn't have such a ``stable" deformation of the curve $V(f_0)$?} That is, what if we didn't know that the origin $\0 \in V(f_0)$ splits into $\delta$ nodes? We can still use the techniques of Theorem 5.3 of \cite{hepmasparam} in this situation. In this case, if $\pi$ parameterizes the deformation of $V(f_0)$, we have
\begin{equation}\label{eqn:generalmilnor}
\mu_\0(f_0) = -m(\0) + \sum_{p \in B_\epsilon \cap V(t-t_0)} \left (\mu_p(f_{t_0}) + m(p) \right )
\end{equation}
where $m(p) := |\pi^{-1}(p)| - 1$; the above formula follows easily from the same proof as Theorem 5.3 of \cite{hepmasparam}. 

%{\color{red} Suppose now that $\U$ is an open neighborhood of the origin in $\C^n$, $\D$ is a small open disk around the origin in $\C$, and $f:(\D \times \U, \0)\rightarrow (\C, 0)$ is a reduced complex analytic function. Thus, $f$ defines a one-parameter analytic family $f_t(\mbf z):=f(t, \mbf z)$.  Additionally, assume $\pi: (\D \times \wt {V(f_0)},\D \times S) \to (V(f) ,\0)$ is the normalization of $V(f)$ (we think of $\pi$ as a parameterization of the ``total" hypersurface $V(f)$), and $S =\pi^{-1}(\0)$ is a finite subset of $\wt {V(f_0)}$, a purely $(n-1)$-dimensional $\Q$-homology (or smooth) manifold that is the normalization of $V(f_0)$. 

\medskip

Suppose now that $\pi_0 : (\wt {V(f_0)},S) \to (V(f_0),\0)$ is the normalization of a (reduced) hypersurface $V(f_0) \subseteq \C^n$, and $\pi$ is a one-parameter unfolding of $\pi_0$ (see \secref{sec:unfold}), so that, if $\D$ is a small open disk around the origin in $\C$,  
$$
\pi : (\D \times \wt {V(f_0)},\{0\} \times S) \to (V(f),\0),
$$
 for some complex analytic function $f \in \cO_{\C^{n+1},\0}$, where $\pi$ is of the form $\pi(t,\z) = (t,\pi_t(\z))$ and $\pi(0,\z) = \pi_0(\z)$.  Here, $S = \pi_0^{-1}(\0)$ is a finite subset of $\wt {V(f_0)}$, a purely $(n-1)$-dimensional $\Q$-homology (or smooth) manifold. We impose this last condition on the normalization of $V(f_0)$ because of the following result regarding the stalk cohomology of the perverse sheaf $\Ndot_X$, defined on any locally reduced, purely $n$-dimensional complex analytic space on which $\Q_X^\bullet[n]$ is perverse. 
 
 Recall that an $n$-dimensional complex analytic space $Y$ is a \textbf{rational homology manifold} (or, a $\Q$-homology manifold) if the natural morphism $\Q_Y^\bullet[n] \to \Idot_Y$ is a quasi-isomorphism \cite{BorhoMac}, where $\Idot_Y$ is the intersection cohomology complex with constant $\Q$-coefficients on $Y$.
 
\begin{thm}[Theorem 2.3 \cite{qhomcriterion}]\label{thm:qhomcriterion}
Let $X$ be a reduced, purely $n$-dimensional complex analytic space on which $\Q_X^\bullet[n]$ is perverse, and let $\pi : Y \to X$ be the normalization of $X$. Then, $Y$ is a $\Q$-homology manifold if and only if $\Ndot_X$ has stalk cohomology concentrated in degree $-n+1$, i.e., $H^k(\Ndot_X)_p = 0$ for all $k \neq -n+1$ and all $p \in X$. 
\end{thm}

The perverse sheaf $\Ndot_X$ is defined in a very straight-forward manner: when $\Q_X^\bullet[n]$ is perverse, there is a natural surjection of perverse sheaves $\Q_X^\bullet[n] \to \Idot_X \to 0$, where $\Idot_X$ is the intersection cohomology complex on $X$ with constant $\Q$ coefficients. Since the category of perverse sheaves is Abelian, this morphism has a kernel, which we define to be $\Ndot_X$. This perverse sheaf, called the \textbf{comparison complex} on $X$, was first defined by the author and Massey in \cite{hepmasparam} (where we originally referred to it as the \textbf{multiple-point complex}), and subsequently studied by the author in \cite{qhomcriterion},\cite{ndotMHMhep} and Massey in \cite{comparison}. $\Ndot_X$ will play a crucial role in this paper as the cohomological generalization of the function $m(p) = |\pi^{-1}(p)|-1$ above.

\begin{rem}\label{rem:coefficients}
Throughout this paper, we will use $\Z$ coefficients when referring to hypersurfaces with smooth normalizations (where \thmref{thm:qhomcriterion} is trivially satisfied), and $\Q$ coefficients when referring to hypersurfaces with $\Q$-homology manifold normalizations. When necessary, we explicitly state which arguments much change (if at all) to change coefficients (see \remref{rem:ipacoefficients}, \remref{rem:lenumcoefficients}).
\end{rem}

 \bigskip

%For ease of notation, we frequently assume $\W$ and $\U$ are of the form $\D^\circ \times \widetilde \W$ or $\D^\circ \times \widetilde \U$ when considering one-parameter unfoldings of parameterizations. See \cite{hepmasparam} and Section~\ref{sec:param}. 

What would it mean to have a generalization of \equref{eqn:generalmilnor}? In the broadest sense, one would want to express numerical data about the singularities of $f_0$ completely in terms of data about the singularities of $f_{t_0}$, for $t_0$ small and non-zero. What changes when we move to higher dimensions?

\medskip

 One of the restrictions in considering parameterizable hypersurfaces $V(f)$ is that they must have codimension-one singularities. In particular, to get the most use out of the complex $\Ndot_{V(f)}$ on $V(f)$, we will assume the image multiple-point set $D = \supp \Ndot_{V(f)} \neq \emptyset$ and $D = \Sigma f$. For parameterized spaces, one always has the inclusion $D \subseteq \Sigma f$, but it is possible for this inclusion to be strict (e.g., if one parameterizes the cusp $y^2=x^3$ in $\C^2$, or more generally, if $V(f)$ itself is a $\Q$-homology manifold). Since $D$ is purely $(n-1)$-dimensional, we are stuck with hypersurfaces that have codimension-one singularities. 
 
 \medskip

 Consequently, we may no longer use the Milnor number in higher dimensions, since this number applies only to isolated singularities. One natural generalization of the Milnor number to higher-dimensional singularities are the \textbf{L\^{e} numbers} $\lambda_{f,\z}^i$, and we will express the L\^{e} numbers of the $t=0$ slice of in terms of the L\^{e} numbers of the $t\neq 0$ slice, together with the \textbf{characteristic polar multiplicities} of $\Ndot_{V(f)}$, which generalize the rank of the hypercohomology group $\hyp^{0}(D\cap F_{t_{|_{V(f)}},\0}; \Ndot_{V(f)})$ used in Theorem 5.1 and Theorem 5.3 of \cite{hepmasparam} (here, $F_{t_{|_{\Sigma f}},\0}$ denotes the Milnor fiber of $t_{|_{\Sigma f}}$ at $\0$, and $D$ denotes the image multiple-point set of $\pi$). This will be explored in Section~\ref{sec:unfold} and Section~\ref{sec:lenums}.

\medskip

When moving to higher dimensions, we must also consider which sort of deformation to allow when relating $f_0$ and $f_{t_0}$ for $t_0$ small and not zero. For this, we choose the notion of a deformation with isolated polar activity (or, an \textbf{IPA-deformation}). Intuitively, these are deformations where the only ``interesting" behavior happens at the origin, and the only change propagates outwards from the origin along curves.  Such deformations exist generically in all dimensions. We examine this notion, first introduced by Massey in \cite{ipadef}, in Section~\ref{sec:IPA} (although an equivalent notion appears as early as 1992 with Massey and Siersma \cite{masseysiersma} under the name \textbf{equi-transversal deformations}, although without the conormal perspective we use here). An ordered tuple of linear forms $\z = (z_0,\cdots,z_k)$ is called an IPA-tuple (for $f$ at $\0$) if, for $1 \leq i \leq k$, $f_{|_{V(z_0,\cdots,z_{i-1})}}$ is an IPA-deformation of $f_{|_{V(z_0,\cdots,z_i)}}$ at $\0$.

In Section~\ref{sec:surface}, we prove the following result. 

\begin{thm}[Theorem ~\ref{thm:main}]
Suppose that $\pi : (\D \times \wt {V(f_0)}, \{0\} \times S) \to (V(f),\0)$ is a one-parameter unfolding of a parameterized hypersurface $\im \pi_0 = V(f_0)$.  Suppose further that $\z = (z_1,\cdots,z_n)$ is chosen such that $\z$ is an IPA-tuple for $f_0 = f_{|_{V(t)}}$ at $\0$. Then, the following formulas hold for the L\^{e} numbers of $f_0$ with respect to $\z$ at $\0$: for $0 < |t_0|  \ll \epsilon \ll 1$,
\begin{align*}
\lambda_{f_0,\z}^0(\0) &= -\lambda_{\Ndot_{V(f_0)},\z}^0(\0) + \sum_{p \in B_\epsilon \cap V(t-t_0)} \left ( \lambda_{{f_{t_0}},\z}^0(p) + \lambda_{\Ndot_{V(f_{t_0})},\z}^0(p) \right ), 
\end{align*}
and, for $1 \leq i \leq n-2$,
\begin{align*}
\lambda_{f_0,\z}^i(\0) &= \sum_{q \in B_\epsilon \cap V(t-t_0,z_1,z_2,\cdots,z_i)} \lambda_{f_{t_0},\z}^{i}(q).
\end{align*}
%\smallskip

In particular, the following relationship holds for $0 \leq i \leq n-2$:
\begin{align*}
\lambda_{f_0,\z}^i(\0) +\lambda_{\Ndot_{V(f_0)},\z}^i(\0) =   \sum_{p \in B_\epsilon \cap V(t-t_0,z_1,z_2,\cdots,z_i)} \left ( \lambda_{{f_{t_0}},\z}^i(p) + \lambda_{\Ndot_{V(f_{t_0})},\z}^i(p) \right ).
\end{align*}

\end{thm}

We then conclude the chapter with some applications of this theorem to various dimensions, and obtain formulas in the same vein as Milnor's double point formula. In particular, we obtain the following result.

\begin{cor}[\corref{cor:surfacedeformation}]
Let $\pi_0: (\C^2,S) \to (\C^3,\0)$ be a finitely-determined map germ parameterizing a surface $V(f_0) \subseteq \C^3$, and let $T,C,$ and $\delta,$ denote the number of triple points, cross caps, and $A_1$-singularities, respectively, appearing in a stabilization of $\pi_0$. Then, the following equality holds:
\begin{align*}
 |\pi_0^{-1}(\0)| -1 = -C+T+\delta +\chi(F_{t_{|_{\Sigma f}},\0}),
\end{align*}
where $F_{t_{|_{\Sigma f}},\0}$ denotes the Milnor fiber of the unfolding parameter of such a stabilization $\im \pi = V(f)$, restricted to the singular locus of $f$ (that is, the complex link of $\Sigma f$ at $\0$).
\end{cor}

\bigskip

\section{IPA-Deformations}\label{sec:IPA}
 
\bigskip

Although we need to consider only the case of a family of parameterized hypersurfaces for this section, much of the machinery we use for Section~\ref{sec:lenums} and Section~\ref{sec:surface} does not require such restrictive hypotheses. That is, the notion of IPA-deformations and L\^{e} numbers (see Massey, \cite{ipadef} and \cite{lecycles}) apply to hypersurface singularities in general, not just parameterized hypersurfaces. 

\medskip

Suppose $\z = (z_0,\cdots,z_n)$ are local coordinates on an open neighborhood $\U \subseteq \C^{n+1}$ of $\0$, so that we have $T^*\U \cong \U \times \C^{n+1}$, with fiber-wise basis $(d_p z_0,\cdots,d_p z_n)$ of $(T^*\U)_p = \tau^{-1}(p)$, where $\tau : T^*\U \to \U$ is the canonical projection map.

\smallskip

Denote by $\Span \langle dz_0,\cdots,dz_k \rangle$ the subset of $T^*\U$ given by $\{(p,\sum_{i=0}^k w_i d_p z_i) \, | p \in \U, w_i \in \C\}$

Let $f:(\U,\0) \to (\C,0)$ be a (reduced) complex analytic function, where $\U$ is a connected open neighborhood of the origin in $\C^{n+1}$.  
%\medskip

Finally, let $\overline{T_f^*\U}$ denote the (closure of) the relative conormal space of $f$ in $\U$, i.e.,
\[
\overline{T_f^*\U} := \overline{\{(p,\xi) \in T^*\U \, | \, \xi(\ker d_p f) = 0\}}.
\]
It is important to note that $\overline{T_f^*\U}$ is a $\C$-conic subset of $T^*\U$, as we will consider its projectivization in Definition~\ref{def:cycles}.

%\smallskip

The following definitions of the relative polar varieties of $f$ differ slightly from their more classical construction  (see, for example \cite{hammlezariski}, \cite{letopuse}, or \cite{leattach}), following that of \cite{ipadef},\cite{numcontrol}. Lastly, the intersection product appearing in the following definitions is that of proper intersections in complex manifolds (See Chapter 6 of \cite{fulton}).

%{\color{red} Perhaps give classical definition first? Referee suggestion.}

\begin{defin}\label{def:relpolcurve}
The \textbf{relative polar curve of $f$ with respect to $z_0$}, denoted $\Gamma_{f,z_0}^1$, is, as an analytic cycle at the origin, the collection of those components of the cycle 
\[
\tau_* \left ( \overline{T_f^*\U} \cdot \im \, dz_0 \right )
\] 
which are not contained in $\Sigma f$, provided that $\overline{T_f^*\U}$ and $\im \, dz_0$ intersect properly in $T^*\U$ (where $\tau_*$ is the proper pushfoward of cycles). 
\end{defin}

%\medskip

 More generally, one can define the higher $k$-dimensional relative polar varieties $\Gamma_{f,\z}^k$ in this manner, by considering the \emph{projectivized} relative conormal space $\PP(\overline{T_f^*\U})$ as follows.  For $0 \leq k \leq n$, consider the subspace $\PP( \Span \langle dz_0,\cdots,dz_k \rangle )$ of $\PP(T^*\U) \cong \U \times \PP^n$, the projectivized cotangent bundle of $\U$ (The following definition \textbf{does not} require one to use the projectivized relative conormal space; we do so to make the formulas involved less cumbersome). 

%\medskip

\begin{defin}\label{def:cycles}
The \textbf{$(k+1)$-dimensional relative polar variety of $f$ with respect to $\z$}, denoted $\Gamma_{f,\z}^k$\,, is, as an analytic cycle at the origin, the collection of those components of 
\begin{align*}
\tau_* \left (\PP(\overline{T_f^*\U}) \cdot \PP\left (\Span \langle dz_0,\cdots,dz_k \rangle \right )\right )
\end{align*}
 which are not contained in the critical locus $\Sigma f$ at the origin, provided that $\PP(\overline{T_f^*\U})$ and $\PP \left (\Span \langle dz_0,\cdots,dz_k \rangle \right)$ intersect properly in $T^*\U$.  By abuse of notation, we also use $\tau$ to denote the canonical projection $\PP(T^*\U) \to \U$.
\end{defin}

See Section \ref{sec:classicalcycles} for the classical definition of $\Gamma_{f,\z}^k$. 

%\medskip

%\begin{rem}
%When considering the relative polar curve, we frequently write $\Gamma_{f,z_0}^1$ in place of $\Gamma_{f,\z}^1$, to highlight the fact that the relative polar curve depends only on the function $f$ and choice of a single linear form $z_0$.
%\end{rem}

%\medskip

Throughout this section (and, this thesis in general), we will use the (shifted) \textbf{nearby and vanishing cycle functors} $\psi_f[-1]$ and $\phi_f[-1]$, respectively, from the bounded derived category $D_c^b(\U)$ of constructible complexes of sheaves on $\U$ to those on $V(f)$ (see for example \cite{kashsch}, \cite{dimcasheaves}, \cite{sgavii1}, or \cite{bbd}). The shifts $[-1]$ are need to for these functors to take perverse sheaves on $\U$ to perverse sheaves on $V(f)$. One of the most important properties of these functors is that, for an arbitrary bounded, constructible complex of sheaves $\Fdot$ on $\U$, we have isomorphisms
\begin{align}\label{eqn:vanishingcycle}
H^k(\psi_f[-1]\Fdot)_p &\cong \hyp^k(F_{f,p};\Fdot) \text{ and } \\
H^k(\phi_f[-1]\Fdot)_p &\cong \hyp^{k+1}(B_\epsilon(p),F_{f,p};\Fdot),
\end{align}
where $\hyp^*$ denotes hypercohomology of complexes of sheaves, and $F_{f,p} = B_{\epsilon}(p) \cap f^{-1}(\xi)$ denotes the \textbf{Milnor fiber of $f$ at $p$} (here $0 < |\xi| \ll \epsilon \ll 1$). The Milnor fiber of a generic linear form $L$ on a space $X$ at a point $p$ is often referred to as the \textbf{complex link} of $X$ at $p$, and we sometimes distinguish this with the notation $\mathbb{L}_{X,p}$. The (stratified) homeomorphism type of $\mathbb{L}_{X,p}$ is independent of the linear form chosen, provided $L$ is sufficiently generic.

If we use $\Fdot = \Z_\U^\bullet[n+1]$ for coefficients, then $\psi_f[-1]$ (resp. $\phi_f[-1]$) recovers the ordinary integral (resp. reduced) cohomology groups of the Milnor fiber $F_{f,p}$ of $f$ at $p$ (up to a shift):
$$
H^k(\phi_f[-1]\Z_\U^\bullet[n+1])_p \cong \wt H^{k+n}(F_{f,p};\Z).
$$
One of the most important properties of the nearby and vanishing cycles are that they fit into a distinguished triangle in the derived category $D_c^b(V(f))$:
$$
(\Fdot)_{|_{V(f)}}[-1]\to \psi_f[-1]\Fdot \to \phi_f[-1]\Fdot \xrightarrow{+1}.
$$
Additionally, the functors $\psi_f[-1]$ and $\phi_f[-1]$ are perverse exact, so this distinguished triangle yields the short exact sequence of perverse sheaves on $V(f)$:
$$
0 \to \Z_{V(f)}^\bullet[n]\xrightarrow{comp}\psi_f[-1]\Z_\U^\bullet[n+1]\xrightarrow{can}\phi_f[-1]\Z_\U^\bullet[n+1]\to 0,
$$
where the morphism $comp$ is known as the \textbf{comparison morphism}, and $can$ the \textbf{canonical morphism}.

\bigskip

We will also make frequent use of the \textbf{microsupport} $SS(\Fdot)$ of a (bounded, constructible) complex of sheaves $\Fdot$ which is a closed $\C^\times$-conic subset of $T^*\U$.  We will use the following characterization of $SS(\Fdot)$ in terms the vanishing cycles (See Prop 8.6.4, of \cite{kashsch}).

\begin{prop}[Microsupport]\label{prop:microsupport}
Let $\Fdot \in D_c^b(\U)$ and let $(p,\xi) \in T^*\U$. Then, the following are equivalent:
\begin{enumerate}
\item $(p,\xi) \notin SS(\Fdot)$. 
%\smallskip

\item There exists an open neighborhood $\Omega$ of $(p,\xi)$ in $T^*\U$ such that, for any $q \in \U$ and any complex analytic function $g$ defined in a neighborhood of $q$ with $f(q) = 0$ and $(q,d_q g) \in \Omega$, one has $(\phi_g \Fdot)_q = 0$. 

\end{enumerate}
\end{prop}

It is instructive to think about the condition $(p,d_p g) \notin SS(\Fdot)$ from the perspective of microlocal/stratified Morse theory. That is, $(p,d_p g) \notin SS(\Fdot)$ if and only if $p$ is not a critical point of $g$ ``with coefficients in $\Fdot$".

Using constant coefficients, we make the following definition to clarify what we mean by a critical point of a function on a possibly singular space.

\begin{defin}\label{def:topcritlocus}
Let $h: (V(f),\0)\to(\C,0)$ be a complex analytic function. We define the \textbf{topological critical locus of $h$} to be the set
$$
\Sigma_{\Top} h := \supp \phi_h[-1]\Z_{V(f)}^\bullet[n] = \tau (\im dh \cap SS(\Z_{V(f)}^\bullet[n]) ).
$$
\end{defin}

\bigskip

 In order to compute numerical invariants associated to certain perverse sheaves (see the characteristic polar multiplicities (Section~\ref{sec:lenums}) and L\^{e} numbers), we need to choose linear forms that ``cut down" the support in a certain way. We now give several equivalent conditions for this ``cutting" procedure, that will be used throughout this paper (see \defref{def:IPAdef}).

%\smallskip
\begin{prop}[IPA-Deformations]\label{prop:IPAequiv}
The following are equivalent:
\begin{enumerate}

\item $\dim_\0 \Gamma_{f,z_0}^1 \cap V(z_0) \leq 0$.

\smallskip

\item $\dim_\0 \Gamma_{f,z_0}^1 \cap V(f) \leq 0. $

\smallskip

\item  $ \dim_{(\0,d_\0 z_0)} \im \, dz_0 \cap (f \circ \tau)^{-1}(0) \cap \overline{T_f^*\U} \leq 0$, where again $\tau: T^*\U \to \U$ is the canonical projection map. 

\smallskip

\item $\dim_{(\0,d_\0 z_0)} SS(\psi_f[-1]\Z_\U^\bullet[n+1]) \cap \im \, dz_0 \leq 0$.

\smallskip

\item $\dim_{(\0,d_\0 z_0)} SS(\Z_{V(f)}^\bullet[n]) \cap \im \, dz_0 \leq 0$.

\smallskip

\item $\dim_\0 \supp \phi_{z_0}[-1]\Z_{V(f)}^\bullet[n] \leq 0$.

\smallskip

\item Away from $\0$, the comparison morphism $\Z_{V(f,z_0)}^\bullet[n-1] \to \psi_{z_0}[-1]\Z_{V(f)}^\bullet[n]$ is an isomorphism. 
\end{enumerate}

\end{prop}

\begin{proof}
The equivalence of statements (1), (2), and (3) are covered in Proposition 2.6  of \cite{ipadef}.
%\smallskip

The equivalence (3) $\iff$ (4) follows directly from the equality 
\[
\overline{T_f^*\U} \cap (f \circ \tau)^{-1}(0) = SS(\psi_f[-1]\Z_\U^\bullet[n+1]).
\]
(See \cite{bmm} for the original result, although the phrasing used above is found in \cite{singenrich}).
%\smallskip

To see the equivalence (4) $\iff (5)$, consider the natural distinguished triangle
\[
 i_*i^*[-1]\Z_{\U}^\bullet[n+1] \to j_!j^!\Z_{\U}^\bullet[n+1] \to \Z_{\U}^\bullet[n+1]  \overset{+1}{\to} \quad (\ddag)
\]
where $i : V(f) \hookrightarrow \U$, and $j : \U \backslash V(f) \hookrightarrow \U$. Then, by \cite{vanaf}, there is an equality of microsupports
\[
SS(\psi_f[-1]\Z_{\U}^\bullet[n+1]) = SS(j_!j^!\Z_{\U}^\bullet[n+1])_{\subseteq V(f)},
\]
where the subscript $\subseteq V(f)$ denotes the union of irreducible components of $SS(j_!j^!\Z_\U[n+1])$ that lie over the hypersurface $V(f)$. But, since $SS(\Z_{\U}[n+1]) \cong \U \times \{\0\}$, $(\ddag)$ implies that 
\[
SS(i_*i^*[-1]\Z_{\U}^\bullet[n+1]) = SS(j_!j^!\Z_{\U}^\bullet[n+1])_{\subseteq V(f)},
\]
by the triangle inequality for microsupports.  The claim follows after noting 
$$
i_*i^*[-1]\Z_{\U}^\bullet[n+1] = \Z_{V(f)}^\bullet[n].
$$

The equivalence $(5) \iff (6)$ follows easily from Proposition~\ref{prop:microsupport}, or see Theorem 3.1 of \cite{critpts}.
%\smallskip

Lastly, one concludes $(6) \iff (7)$ trivially from the short exact sequence of perverse sheaves
\[
0 \to \Z_{V(f,z_0)}^\bullet[n-1] \to \psi_{z_0}[-1]\Z_{V(f)}^\bullet[n] \to \phi_{z_0}[-1]\Z_{V(f)}^\bullet[n] \to 0
\]
on $V(f,z_0)$. 

\end{proof}

\begin{rem}[$\Z$ vs. $\Q$ coefficients]\label{rem:ipacoefficients}
As we mentioned in the introduction of this section, all results hold with either $\Z$ coefficients or $\Q$ coefficients (depending on whether the normalization of $V(f)$ is smooth, or a $\Q$-homology manifold). To see this for \propref{prop:IPAequiv}, suppose that $\dim \supp \phi_L[-1]\Q_{V(f)}^\bullet[n] \leq 0$ but $\dim \supp \phi_L[-1]\Z_{V(f)}^\bullet[n] > 0$. Then, at a generic point $p$ of $\supp \phi_L[-1]\Z_{V(f)}^\bullet[n]$, the stalk cohomology of $\phi_L[-1]\Z_{V(f)}^\bullet[n]$ is a torsion $\Z$-module concentrated in a single degree by the perversity of $\phi_L[-1]\Z_{V(f)}^\bullet[n]$. However, this cohomology must be free Abelian (see e.g, L\^{e}'s classical result about the cohomology of the Milnor fiber, or Proposition 1.2.3 of \cite{hepmasparam}) and is therefore zero. The reverse implication, from $\Z$ to $\Q$ coefficients, is trivial.

Thus, $\dim_\0 \supp \phi_L[-1]\Q_{V(f)}^\bullet[n] \leq 0$ if and only if $\dim_\0 \supp \phi_L[-1]\Z_{V(f)}^\bullet[n] \leq 0$. 
\end{rem}

%\medskip

\begin{defin}\label{def:IPAdef}
Given an analytic function $f : (\U,\0) \to (\C,0)$ and a non-zero linear form $z_0 : (\U,0) \to (\C,0)$, we say that $f$ is a deformation of $f_{|_{V(z_0)}}$ with \textbf{isolated polar activity} at $\0$ (or, an \textbf{IPA-deformation} for short) if the equivalent statements of Proposition~\ref{prop:IPAequiv} hold. 
\end{defin}

\begin{rem}\label{rem:IPAvsprepolar}
IPA-deformations are closely related to the notion of prepolar deformations \cite{prepolardef}; given a Thom $a_f$ stratification  $\mathfrak{S}$ of $V(f)$ and linear form $L$, we say $f$ is a\textbf{ prepolar deformation} of $f_{|_{V(L)}}$ if $V(L)$ transversally intersects all strata $S \in \mathfrak{S} \backslash \{\0\}$ in a neighborhood of the origin. We can alternatively phrase this as 
$$
\dim_\0 \bigcup_{S \in \mathfrak{S}} \Sigma \left ( L_{|_S} \right ) \leq 0,
$$
where the union $\bigcup_{S \in \mathfrak{S}} \Sigma \left (L_{|_S} \right ) =: \Sigma_{\mathfrak{S}} L_{|_{V(f)}}$ is called the \textbf{stratified critical locus of $L_{|_{V(f)}}$ with respect to $\mathfrak{S}$} (see Definition 1.3 of \cite{critpts}) .

In particular, a prepolar deformation is defined with respect to a given $a_f$ stratification $\mathfrak{S}$, whereas an IPA-deformation does not refer to any stratification. While one does always have the inclusion
$$
\supp \phi_L[-1]\Z_{V(f)}^\bullet[n] =: \Sigma_{\Top} \left ( L_{|_{V(f)}} \right ) \subseteq \Sigma_{\mathfrak{S}} \left (L_{|_{V(f)}} \right )
$$
(this follows by stratified Morse theory, or Remark 1.10 of \cite{critpts}, or Proposition 8.4.1 and Exercise 8.6.12 of \cite{kashsch})), it is an open question whether or not there exist IPA deformations that are not prepolar deformations.
\end{rem}

We can iterate the notion of an IPA-deformation as follows.
\begin{defin}\label{def:IPAtuple}
Let $k \geq 0$. A $(k+1)-$tuple $(z_0,\cdots,z_k)$ is said to be an IPA-tuple for $f$ at $\0$ if, for all $1 \leq i \leq k$, $f_{|_{V(z_0,\cdots,z_{i-1})}}$ is an IPA-deformation of $f_{|_{V(z_0,\cdots,z_i)}}$ at $\0$.
\end{defin}

%\bigskip

The following lemma follows from an inductive application of Theorem 1.1 of \cite{relcoresult}, and is crucial for our understanding of what IPA-deformation ``looks like" in the cotangent bundle (cf. Proposition~\ref{prop:IPAequiv}, item (2)).

\begin{lem}[Gaffney, Massey, \cite{GaffneyTerence2018Alfr}]\label{lem:IPAiterate2}
Let $k \geq 0$. Then, for all $p \in V(z_0,\cdots,z_{k-1})$ with $d_p z_k \notin \left (\overline{T_{f_{|_{V(z_0,\cdots,z_{k-1})}}}^*V(z_0,\cdots,z_{k-1})}\right)_p $ , we have 
 \[
  \left (\overline{T_f^*\U} \right )_p \cap \textnormal{Span}\langle d_p z_0,\cdots,d_p z_k \rangle = 0.
 \] 
\end{lem}

%{\color{red}
%\begin{rem}
%We want to briefly explain the motivation of Lemma~\ref{lem:IPAiterate2}. If $f$ is an IPA-deformation of $f_{|_{V(z_0)}}$ at $\0$, then for all $p \in V(f,z_0)$ away from $\0$ we know by Proposition~\ref{prop:IPAequiv} that we have
%\[
%d_pz_0 \notin SS(\psi_f[-1]\Z_\U^\bullet[n+1]).
%\]	
%XThat is, the cohomology of $\psi_f[-1]\Z_\U^\bullet[n+1]$ is unchanged when one moves in the direction of $dz_0$ from points $p$ away from the origin. 

%\smallskip

%Although it takes a lot more familiarity with the derived category, one can further show that, at such a point $p$, one has a isomorphism
%\[
%\psi_{f_{|_{V(z_0)}}}[-1]\Z_{V(z_0)}^\bullet[n] \cong \psi_{{z_0}_{|_{V(f)}}}[-1]\psi_f[-1]\Z_\U^\bullet[n+1]
%\]
%(see Proposition 3.1 of \cite{relcoresult}). Strikingly, this implies the equality 
%\[
%SS(\psi_f[-1]\Z_\U^\bullet[n+1])\cap ((f,z_0)\circ \tau)^{-1}(0) = SS(\psi_{f_{|_{V(z_0)}}}[-1]\Z_{V(z_0)}^\bullet[n]) \cap ((f,z_0)\circ \tau)^{-1}(0)
%\]
%over points $p \neq \0$. 
%%}
%\bigskip

The main goal of this subsection is the following result. This result, originally from \cite{lecycles}, is presented here with the ``weaker'' hypothesis of choosing an IPA-tuple, in lieu of a prepolar-tuple. For the definition of the L\^{e} numbers of $f$ with respect to a tuple of linear forms $\z$, see Section \ref{sec:classicalcycles}. %\secref{sec:appendix}.

\begin{prop}[Existence of L\^{e} Numbers of a Slice]\label{prop:basic}
Suppose that $\z = (z_0,\cdots,z_n)$ is an IPA-tuple for $f$ at $\0$, and use coordinates $\widetilde{\z} = (z_1,\cdots,z_n)$ for $V(z_0)$.
Then, for $0 \leq i \leq \dim_\0 \Sigma f$, the L\^{e} numbers $\lambda_{f,\z}^i(\0)$ are defined, and the following equalities hold:
\begin{align*}
\lambda_{f_{|_{V(z_0)}},\widetilde{\z}}^0(\0) &= \left(\Gamma_{f,z_0}^1\cdot V(z_0)\right)_\0 + \lambda_{f,\z}^1(\0) \\
\lambda_{f_{|_{V(z_0)}},\widetilde{\z}}^i(\0) &= \lambda_{f,\z}^{i+1}(\0),
\end{align*}
for $1 \leq i \leq \dim_\0 \Sigma f-1$, where $\Gamma_{f,z_0}^1$ is the relative polar curve of $f$ with respect to $z_0$.
\end{prop}
\begin{proof}
The proof follows Theorem 1.28 of \cite{lecycles}, \emph{mutatis mutandis} (changing prepolar to IPA and working with covectors instead of tangent hyperplanes). 

%\smallskip

Via the Chain Rule, it suffices to demonstrate that 
$$
\dim_\0 \Gamma_{f,\z}^{i+1} \cap V(f) \cap V(z_0,\cdots,z_{i-1}) \leq 0,
$$
since any analytic curve in $\Gamma_{f,\z}^{i+1} \cap V \left ( \frac{\partial f}{\partial z_i} \right ) \cap V(z_0,\cdots,z_{i-1})$ passing through $\0$ must be contained in $V(f)$, where $\Gamma_{f,\z}^{i+1}$ is the $(i+1)$-dimensional relative polar variety of $f$ with respect to $\z$ (\defref{def:cycles}). 

Suppose that we had a sequence of points $p \in  \Gamma_{f,\z}^{i+1} \cap V(f) \cap V(z_0,\cdots,z_{i-1})$ approaching $\0$. As each $p$ is contained in $\Gamma_{f,\z}^{i+1}$, for each $p$ we can find a sequence $p_k \to p$ with $p_k \notin \Sigma f$ satisfying $\langle d_{p_k} f \rangle \subseteq \text{Span}\langle d_{p_k} z_0,\cdots,d_{p_k} z_{i-1} \rangle$ for each $k$. But then, by construction, we have found a nonzero element in the intersection $\left ( \overline{T_f^*\U} \right )_p \cap \text{Span}\langle d_p z_0\cdots,d_p z_{i-1} \rangle$, contradicting Lemma~\ref{lem:IPAiterate2}. 

\end{proof}

\section{Unfoldings and $\Ndot_{V(f)}$}\label{sec:unfold}

As mentioned in the introduction of this section, we will be considering parameterized hypersurfaces that are the total space of a family of parameterized hypersurfaces. We make this precise with the following definition. 

\begin{defin}\label{def:unfolding}
A parameterization $\pi :(\D \times \wt {V(f_0)},\{0\} \times S) \to (V(f),\0)$ is said to be a \textbf{one-parameter unfolding} with unfolding parameter $t$ if $\pi$ is of the form
$$
\pi(t,\z) = (t,\pi_t(\z))
$$
where $\pi_0(\z) := \pi(0,\z)$ is a generically one-to-one parameterization of $V(f,t)$. 
\end{defin}

We say that a parameterization $\pi_0$ has an \textbf{isolated instability} at $\0$ with respect to an unfolding $\pi$ of $\pi_0$ with parameter $t$ if one has $\dim_\0 \Sigma_{\textnormal{top}} t_{|_{\im \pi}} \leq 0$. Compare this with the more general (standard) notion in \secref{sec:singmaps}.

The following proposition is one of our main motivations for using IPA-deformations: they naturally appear from one-parameter unfoldings with isolated instabilities. 

\begin{prop}\label{prop:IPAexistence}
Suppose $\pi : (\D \times \wt {V(f_0)},\{0\} \times S) \to (V(f),\0)$ is a 1-parameter unfolding of $\pi_0$ with unfolding parameter $t$, such that $\pi_0$ has an isolated instability at $\0$ with respect to $\pi$. Then, $f$ is an IPA-deformation of $f_{|_{V(t)}}$ at $\0$. 
\end{prop}

\begin{proof}
 By definition, $\pi_0$ has an isolated instability at $\0$ with respect to the unfolding $\pi$ with parameter $t$ if 
$$
\dim_\0 \Sigma_{\textnormal{top}} \left (t_{|_{V(f)}} \right ) \leq 0.
$$
 Following Definition 1.9 of \cite{critpts}, 
\begin{align*}
 \Sigma_{\textnormal{top}} \left (t_{|_{V(f)}} \right ) &= \overline{\{ p \in V(f) \, | \, (p,d_p t) \in SS(\Z_{V(f)}^\bullet[n])\}} \\
 &= \tau \left (SS(\Z_{V(f)}^\bullet[n]) \cap \im dt \right ),
 \end{align*}
 where $\tau: T^*\U \to \U$ is the natural projection. This follows immediately from Proposition~\ref{prop:microsupport}.  
%\smallskip
 
 Consequently, if $\dim_\0 \Sigma_{\textnormal{top}} \left (t_{|_{V(f)}} \right )\leq 0$, it follows that $(\0,d_\0 t)$ is an isolated point in the intersection $SS(\Z_{V(f)}^\bullet[n]) \cap \im dt$, and the  the result follows by Proposition~\ref{prop:IPAequiv}.
\end{proof}

\begin{rem}\label{rem:finitelydeterminedIPA}
It is well-known that finitely-determined map germs $\pi_0$ have isolated instabilities with respect to a generic one-parameter unfolding (\cite{MatherJohnN1973GP} pg. 241, and \cite{1976POFD}). Consequently, generic one-parameter unfoldings of finitely-determined maps parameterizing hypersurfaces all give IPA-deformations. We shall make use of this fact later in \secref{sec:surface}.
\end{rem}

\begin{rem}\label{rem:ndotslicing}
If $\pi$ is a one-parameter unfolding of a parameterization $\pi_0$, then for all $t_0$ small, it is easy to see that there is an isomorphism ${\Ndot_{V(f)}}_{|_{V(t-t_0)}}[-1] \cong \Ndot_{V(f_{t_0})}$, where $\pi_{t_0}(\mbf z) = \pi(t_0,\mbf z)$. 
\end{rem}

%\bigskip

\begin{exm}
 In the situation of Milnor's double-point formula,  $\pi:(\D \times \C, \{0\} \times S) \to (\C^3,\0)$ parameterizes the deformation of the curve $V(f_0)$ with $r$ irreducible components at $\0$ into a curve $V(f_{t_0})$ with only double-point singularities.  Hence, $\dim_\0 V(f) = 2$, and the image multiple-point set $D$ is purely 1-dimensional at $\0$. 
 
% \medskip
 
 Since $\pi$ is a one-parameter unfolding with parameter $t$, we moreover have 
\[
 {\Ndot_{V(f)}}_{|_{V(t-t_0)}}[-1] \cong \Ndot_{V(f_{t_0})},
 \]
 where $\Ndot_{V(f_{t_0})}$ is the multiple-point complex of the parameterization $\pi_{t_0}(\z)$. For all $t_0 \neq 0$ small, $\Ndot_{V(f_{t_0})}$ is supported on the set of double points of $V(f_{t_0})$, and at each such double-point $p$ we have $\rank H^0(\Ndot_{V(f_0)})_p = |\pi^{-1}(p)|-1 = 1$. 
 
% \smallskip
 
 At $\0 \in V(f_0)$, we have $\pi^{-1}(\0) = S$, and $|S| = r$ by assumption. Thus, $\rank H^0(\Ndot_{V(f_0)})_\0 = r-1$. 
\end{exm}

\section{Characteristic Polar Multiplicities}\label{sec:lenums}
The central concept of this section, the characteristic polar multiplicities of a perverse sheaf, were first defined and explored in \cite{numinvar}. These multiplicities, defined with respect to a ``nice" choice of a tuple of linear forms $\z = (z_0,\cdots,z_s)$, are non-negative, integer-valued functions that mimic the properties of the L\^{e} numbers associated to non-isolated hypersurface singularities (see \cite{lecycles}), and the characteristic polar multiplicities of the vanishing cycles $\phi_f[-1]\Z_\U^\bullet[n+1]$ with respect to $\z$ coincide with the L\^{e} numbers of $f$ with respect to $\z$.

%\smallskip

%Recall that the (shifted) nearby and vanishing cycle functors take perverse sheaves to perverse sheaves, and at an isolated point in the support of a perverse sheaf, the only possible non-zero stalk cohomology is in degree zero.

\begin{defin}[Corollary 4.10 \cite{numinvar}]\label{def:charpol}
Let $\Pdot$ be a perverse sheaf on $V(f)$, with $\dim_\0 \supp \Pdot = s$. Let $\z = (z_0,\cdots,z_s)$ be a tuple of linear forms such that, for all $0 \leq i \leq s$, we have
\[
\dim_\0 \supp \phi_{z_i-z_i(p)}[-1]\psi_{z_{i-1}-z_{i-1}(p)}[-1]\cdots\psi_{z_0-z_0(p)}[-1]\Pdot \leq 0.
\]
Then, the $i$-dimensional \textbf{characteristic polar multiplicity of $\Pdot$} with respect to $\z$ at $p \in V(g)$ is defined and given by the formula
\[
\lambda_{\Pdot,\z}^i(p) = \rank_\Z H^0(\phi_{z_i-z_i(p)}[-1]\psi_{z_{i-1}-z_{i-1}(p)}[-1]\cdots\psi_{z_0-z_0(p)}\Pdot)_p.
\]
\end{defin}

%\medskip

%\medskip

\begin{rem}
In general, one can define the characteristic polar multiplicities of any object in the bounded, derived category of constructible sheaves on $V(f)$, but they are slightly more cumbersome to define, and no longer need to be non-negative.
\end{rem}

%\medskip

\begin{exm}\label{exm:lecharmult}
Let $f :\U \to \C$ be an analytic function, with $f(\0) = 0$, $\U$ an open neighborhood of the origin in $\C^{n+1}$, and $\dim_\0 \Sigma f = s$. Then, $\phi_f[-1]\Z_\U^\bullet[n+1]$ is a perverse sheaf on $V(f)$, with support equal to $\Sigma f \cap V(f)$.  Indeed, the containment $\supp \phi_f[-1]\Z_\U^\bullet[n+1] \subseteq \Sigma f \cap V(f)$ follows from the complex analytic Implicit Function Theorem. For the reverse containment, if $p \notin \supp \phi_f[-1]\Z_\U^\bullet[n+1]$, then the Milnor monodromy on the nearby cycles is the identity morphism, so that the Lefschetz number of the monodromy cannot be zero; by A'Campo's result \cite{acamp}, we therefore have $p \notin V(f) \cap \Sigma f$.

We then have 
\[
\lambda_{f,\z}^i(p) = \lambda_{\phi_f[-1]\Z_{\U}^\bullet[n+1],\z}^i(p)
\]
for all $0 \leq i \leq s$, and all $p$ in an open neighborhood of $\0$ \cite{numinvar}. %For the purposes of this paper, we can take this as the \textbf{definition} of the L\^{e} numbers (cf. \secref{sec:appendix}). 
\end{exm}

\begin{rem}[$\Z$ vs. $\Q$ Coefficients]\label{rem:lenumcoefficients}
By Massey (Theorem 3.4 \cite{MR2215774}), if $\dim_\0 \Sigma f = s$, then there is a chain complex of \textbf{free} Abelian groups
$$
0 \xrightarrow{\partial_{s+1}} \Z^{\lambda_{f,\z}^s(p)} \xrightarrow{\partial_s} \Z^{\lambda_{f,\z}^{s-1}(p)} \xrightarrow{\partial_{s-1}} \cdots \xrightarrow{\partial_2} \Z^{\lambda_{f,\z}^1(p)} \xrightarrow{\partial_1} \Z^{\lambda_{f,\z}^0(p)} \xrightarrow{\partial_0} 0
$$
satisfying $\ker \partial_j / \im \partial_{i+1} \cong \wt H^{n-j}(F_{f,\0};\Z)$. Since this complex is free, tensoring this complex with $\Q$ will compute $\wt H^{n-j}(F_{f,\0};\Q)$. Hence, we can use either $\Z$ or $\Q$ coefficients in when characterizing the L\^{e} numbers $\lambda_{f,\z}^i(p)$ in terms of the characteristic polar multiplicities of the vanishing cycles. 

\end{rem}
%\medskip

\begin{exm}
If $\dim_\0 \Sigma f = 0$, any non-zero linear form $z_0$ suffices for this construction, since $\psi_{z_0}[-1]\phi_f[-1]\Z_{\U}^\bullet[n+1] = 0$. Then, the only non-zero L\^{e} number of $f$ is $\lambda_{f,z_0}^0(\0)$, and we have
\begin{align*}
\lambda_{f,z_0}^0(\0) &= \rank_\Z H^0(\phi_{z_0}[-1]\phi_f[-1]\Z_{\U}^\bullet[n+1])_\0 \\
&= \rank_\Z H^0(\phi_f[-1]\Z_\U^\bullet[n+1])_\0 \\
&= \text{Milnor number of $f$ at $\0$}.
\end{align*}
 
\end{exm}

%\medskip

\begin{exm}
If $\dim_\0 \Sigma f = 1$, we need $z_0$ such that $\dim_\0 \Sigma \left (f_{|_{V(z_0)}} \right ) = 0$, and any non-zero linear form suffices for $z_1$.
Then the only non-zero L\^{e} numbers of $f$ with respect to $\z = (z_0,z_1)$ are $\lambda_{f,\z}^0(\0)$ and $\lambda_{f,\z}^1(p)$ for $p \in \Sigma f$. At $\0$, we have
\begin{align*}
\lambda_{f,\z}^1(\0) &= \rank H^0(\phi_{z_1}[-1]\psi_{z_0}[-1]\phi_f[-1]\Z_\U^\bullet[n+1])_\0  \\
&= \sum_{C \subseteq \Sigma f \text{ irr.comp. at $\0$}} \overset{\circ}{\mu}_C \left (C \cdot V(z_0) \right )_\0,
\end{align*}
where $\overset{\circ}{\mu}_C$ denotes the generic transverse Milnor number of $f$ along $C \backslash \{0\}$. 
\end{exm}

\begin{rem}\label{rem:charpolcycleeuler}
Analogous to the L\^{e} numbers $\lambda_{f,\z}^i(p)$, the characteristic polar multiplicities of a perverse sheaf may be expressed as intersection numbers. That is, suppose we have a perverse sheaf $\Pdot$ and a tuple of linear forms $\z$ such that, for all $0 \leq i \leq \dim_\0 \supp \Pdot$, the characteristic polar multiplicities $\lambda_{\Pdot,\z}^i(p)$ are defined for all $p$ in a neighborhood $\U$ of $\0$.  Then, there is a unique collection of non-negative analytic cycles $\Lambda_{\Pdot,\z}^i$ called the \textbf{characteristic polar cycles} of $\Pdot$ with respect to $\z$ satisfying, for all $p \in \U$,
\[
\lambda_{\Pdot,\z}^i(p) = \left ( \Lambda_{\Pdot,\z}^i \cdot V(z_0-p_0,\cdots,\cdots,z_{i-1}-p_{i-1}) \right )_p.
\]
 These cycles can also be thought of as being defined by the constructible function $\chi(\Pdot)_p$, so that
$$
\chi(\Pdot)_p := \sum_i (-1)^i H^i(\Pdot)_p = \sum_i (-1)^i \lambda_{\Pdot,\z}^i(p).
$$
\end{rem}

\begin{exm}\label{exm:ndottriplepoint}
To illustrate this method of computing the characteristic polar multiplicities, we will compute $\lambda_{\Ndot_{V(f)},\z}^0(\0)$ and $\lambda_{\Ndot_{V(f)},\z}^1(\0)$ for a triple point singularity in $\C^3$, e.g., $V(f) = V(xyz)$. Clearly $V(f)$ is parameterized (the normalization $\pi$ of $V(xyz)$ separates the three planes into a disjoint union in three copies of $\C^3$), and so $\Ndot_{V(f)}$ has stalk cohomology concentrated in degree $-1$, implying 
$$
\chi(\Ndot)_p = -|\pi^{-1}(p)| + 1.
$$
Away from the origin, on the singular locus of $V(xyz)$, $\chi(\Ndot_{V(f)})_p$ has value $-1$ everywhere, and so we can identify the $1$-dimensional characteristic polar cycle of $\Ndot_{V(f)}$ as the sum of the lines of intersection of these three planes, each weighted by 1. Thus, $\lambda_{\Ndot_{V(f)},\z}^1(\0) = 3$. Since $\chi(\Ndot_{V(f)})_\0 = -2$, we find that $\lambda_{\Ndot_{V(f)},\z}^0(\0) = 1$, from the equality
$$
-2 = \chi(\Ndot_{V(f)})_\0 = \lambda_{\Ndot_{V(f)},\z}^0(\0) - \lambda_{\Ndot_{V(f)},\z}^1(\0) = \lambda_{\Ndot_{V(f)},\z}^0(\0) -3.
$$
\end{exm}

\begin{rem}\label{rem:charpolcycle}
We will need the representation of the characteristic polar multiplicities as intersection numbers in Section~\ref{sec:surface} when we will use the dynamic intersection property for proper intersections to understand $\lambda_{\Ndot_{V(f)},\z}^i(0)$. By this, we mean the equality 
$$
\left ( \Lambda_{\Pdot,\z}^i \cdot V(z_0,z_1,\cdots,z_{i-1}) \right )_\0 = \sum_{p \in B_\epsilon \cap \Lambda_{\Pdot,\z}^i \cap V(z_0-t)} \left (\Lambda_{\Pdot,\z}^i \cdot V(z_0-t,z_1,\cdots,z_{i-1}) \right )_p
$$
for $0 < |t| \ll \epsilon \ll 1$ (see chapter 6 of \cite{fulton}). Additionally, we will make use of the fact that characteristic polar multiplicities of perverse sheaves are additive on short exact sequences in \secref{sec:surface}. Precisely, if 
$$
0 \to \Adot \to \Bdot \to \Cdot \to 0
$$
is a short exact sequence of perverse sheaves, and if coordinates $\z$ are generic enough so that $\lambda_{\Bdot,\z}^i(p)$ is defined, then $\lambda_{\Adot,\z}^i(p)$ and $\lambda_{\Cdot,\z}^i(p)$ are defined, and
$$
\lambda_{\Bdot,\z}^i(p) = \lambda_{\Adot,\z}^i(p) + \lambda_{\Cdot,\z}^i(p).
$$
(See Proposition 3.3 of \cite{numinvar}.)
\end{rem}

%\medskip

%\bigskip

\begin{lem}\label{lem:IPAimplies1}
If $\pi$ is a one-parameter unfolding (with parameter $t$) of a parameterization of $V(f,t)$ with isolated instability at the origin, then the $0$-dimensional characteristic polar multiplicity of $\Ndot_{V(f)}$ with respect to $t$ is defined, and
\[
\lambda_{\Ndot_{V(f)},t}^0(\0) = \lambda_{\Z_{V(f)}^\bullet[n],t}^0(\0) = \left (\Gamma_{f,t}^1 \cdot V(t) \right )_\0.
\]
\end{lem}

\begin{proof}
If $f$ is an IPA-deformation of $f_{|_{V(t)}}$ at $\0$, then $\dim_\0 \supp \phi_t[-1]\Z_{V(f)}^\bullet[n] \leq 0$, by Proposition~\ref{prop:IPAequiv}. By Definition~\ref{def:charpol}, this is precisely what is needed to define $\lambda_{\Z_{V(f)}^\bullet[n],t}^0(\0)$. Then, by a proper base-change, we have $\phi_t \pi_* \cong \hat{\pi}_*\phi_{t \circ \pi}$, where $\hat{\pi} : V(t \circ \pi) \to V(f,t)$ is the pullback of $\pi$ via the inclusion $V(f,t) \hookrightarrow V(f)$. But, because $\pi$ is a one-parameter unfolding, $t \circ \pi$ is a linear form on affine space and has no critical points; hence, $\phi_{t \circ \pi}\Z_\U^\bullet = 0$. 

%\smallskip

Consequently, it follows from the short exact sequence of perverse sheaves
\[
0 \to \phi_t[-1]\Ndot_{V(f)} \to \phi_t[-1]\Z_{V(f)}^\bullet[n] \to \phi_t[-1]\pi_*\Z_{\D \times \wt X}^\bullet[n] \to 0
\]
that there is an equality $\lambda_{\Ndot_{V(f)},t}^0(\0) = \lambda_{\Z_{V(f)}^\bullet[n],t}^0(\0)$, since the characteristic polar multiplicities are additive on short exact sequences. 

It is then a classical result by L\^e, Hamm, Teissier, and Siersma that, for sufficiently generic $t$, 
\[
\lambda_{\Z_{V(f)}^\bullet[n],t}^0(\0) = \left (\Gamma_{f,t}^1 \cdot V(t) \right )_\0;
\]
the result in the generality of IPA-deformations is found in \cite{gencalcvan}. The claim follows.
\end{proof}
%\medskip

\begin{rem}
The unfolding condition is not needed for the characteristic polar multiplicities of $\Ndot_{V(f)}$ to be defined, but it \textbf{is needed} for the vanishing $\lambda_{\pi_*\Z_{\D \times \wt {V(f_0)}}^\bullet[n],t}^0(\0) = 0$ which yields the equalities of \lemref{lem:IPAimplies1}.
\end{rem}

\begin{exm}\label{exm:ndotwhitneyumbrella}
Let us compute $\lambda_{\Ndot_{V(f)},t}^0(\0)$ in the case where $V(f)$ is the Whitney umbrella, with defining function $f(x,y,t) = y^2-x^3-tx^2$. Then, we can realize $V(f)$ as the total space of the one-parameter unfolding $\pi(t,u) = (u^2-t,u(u^2-t),t)$ with parameter $t$, and \lemref{lem:IPAimplies1} tells us that $\lambda_{\Ndot_{V(f)},t}^0(\0)$ is equal to the intersection multiplicity $\left (\Gamma_{f,t}^1 \cdot V(t) \right )_\0$. A quick computation tells us that the relative polar curve $\Gamma_{f,t}^1$ is equal to $V(3x+2t,y)$, and thus transversely intersects $V(t)$ at $\0$. Hence,
$$
\lambda_{\Ndot_{V(f)},t}^0(\0) = \left ( \Gamma_{f,t}^1 \cdot V(t) \right )_\0 = 1.
$$
\end{exm}

The iterated IPA-condition implies the higher characteristic polar multiplicities of $\Ndot_{V(f)}$ exist as well. 
\begin{thm}\label{thm:IPAimplies2}
Suppose that $(t,\z) = (t,z_1,\cdots,z_n)$ is an IPA-tuple for $g$ at $\0$.  Then, for $0 \leq i \leq n-1$, the characteristic polar multiplicities $\lambda_{\Ndot_{V(f)},(t,\z)}^i(\0)$ with respect to $(t,\z)$ are defined, and the following equalities hold:
\[
\lambda_{\Ndot_{V(f_0)},\z}^0(\0) = \lambda_{\Ndot_{V(f)},(t,\z)}^1(\0) - \lambda_{\Ndot_{V(f)},(t,\z)}^0(\0),
\]
and, for $1 \leq i \leq n-2$, 
\[
\lambda_{\Ndot_{V(f_0)},\z}^i(\0) = \lambda_{\Ndot_{V(f)},(t,\z)}^{i+1}(\0).
\]
\end{thm}

\begin{proof}
That $\lambda_{\Ndot_{V(f)},(t,\z)}^0(\0)$ is defined is precisely the inequality $\dim_\0 \supp \phi_t[-1]\Ndot_{V(f)} \leq 0$ concluded in Lemma~\ref{lem:IPAimplies1} from the inclusion of perverse sheaves
\[
0 \to \phi_t[-1]\Ndot_{V(f)} \to \phi_t[-1]\Z_{V(f)}^\bullet[n] 
\]
By Proposition 3.2 of \cite{numinvar}, it remains to show that $\lambda_{\Ndot_{V(f)},(t,\z)}^i(\0)$ is defined for $1 \leq i \leq n-1$, i.e., we need to show that
\[
\dim_\0 \supp \phi_{z_{i-1}}[-1]\psi_{z_{i-2}}[-1]\cdots \psi_{z_1}[-1]\psi_t[-1]\Ndot_{V(f)} \leq 0.
\]
From the inclusion of perverse sheaves
\[
0 \to \Ndot_{V(f)} \to \Z_{V(f)}^\bullet[n],
\]
it follows that $\lambda_{\Ndot_{V(f)},(t,\z)}^i(\0)$ is defined if $\lambda_{\Z_{V(f)}^\bullet[n],(t,\z)}^i(\0)$ is defined, by the triangle inequality for supports of perverse sheaves. 
%\medskip

Since $(t,\z)$ is an IPA-tuple for $f$ at $\0$, Proposition~\ref{prop:IPAequiv} gives, for $1 \leq i \leq n-1$,
\[
\dim_\0 \supp \phi_{z_i}[-1]\Z_{V(f,t,z_1,\cdots,z_{i-1})}^\bullet[n-i] \leq 0.
\]
Thus, away from $\0$, each of the comparison morphisms
\begin{align*}
\Z_{V(f,t,z_1,\cdots,z_{i-1},z_i)}^\bullet[n-i-1] \overset{\sim}{\to} \psi_{z_i}[-1]\Z_{V(f,t,z_1,\cdots,z_{i-1})}^\bullet[n-i]
\end{align*}
is an isomorphism for $1 \leq i \leq n-1$. Consequently, 
\[
\dim_\0 \supp \phi_{z_i}[-1]\Z_{V(f,t,z_1,\cdots,z_{i-1})}^\bullet[n-i] \leq 0
\]
implies 
\[
\dim_\0 \supp \phi_{z_{i-1}}[-1]\psi_{z_{i-2}}[-1]\cdots \psi_{z_1}[-1]\psi_t[-1]\Z_{V(f)}^\bullet[n] \leq 0,
\]
and the claim follows.
%\bigskip

\end{proof}

\begin{rem}
In the wake of a recent result \cite{comparison} by David Massey, we can obtain a much simpler proof of the above result; one has the identification $\Ndot_{V(f)} \cong \ker \{ \Id-\wt T_f\}$ for hypersurfaces, where $\wt T_f$ is the Milnor monodromy automorphism on the vanishing cycles $\phi_f[-1]\Z_\U^\bullet[n+1]$, with the kernel being taken in the category of perverse sheaves on $V(f)$. Consequently, $\Ndot_{V(f)}$ is a perverse subobject of the vanishing cycles $\phi_f[-1]\Z_\U^\bullet[n+1]$, and we obtain \thmref{thm:IPAimplies2} by either the triangle inequality for microsupports, or the fact that characteristic polar multiplicities are additive on short exact sequences \cite{numinvar} from the fact that $\supp \Ndot_{V(f)} \subseteq \supp \phi_f[-1]\Z_\U^\bullet[n+1] = V(f) \cap \Sigma f$.

We still wish to include our original proof of \thmref{thm:IPAimplies2}, since the methods used provide good intuition for how one uses IPA-deformations cohomologically to ``move around" the origin.
\end{rem}

\section{Milnor's Result and Beyond}\label{sec:surface}

We wish to express the L\^{e} numbers of $f_0$ entirely in terms of data from the L\^{e} numbers of $f_{t_0}$ and the characteristic polar multiplicities of both $\Ndot_{V(f_0)}$ and $\Ndot_{V(f_{t_0})}$, for $t_0$ small and nonzero. The starting point is Proposition~\ref{prop:basic}:
\begin{align*}
\lambda_{f_0,\z}^0(\0) &= \left ( \Gamma_{f,t}^1 \cdot V(t) \right )_\0 + \lambda_{f,(t,\z)}^1(\0) \\
\lambda_{f_{t_0},\z}^i(\0) &= \lambda_{f,(t,\z)}^{i+1}(\0),
\end{align*}
where $(t,\z) = (t,z_1,\cdots,z_n)$ is an IPA-tuple for $f$ at $\0$.  From Lemma~\ref{lem:IPAimplies1}, we have $\left ( \Gamma_{f,t}^1 \cdot V(t) \right )_\0 = \lambda_{\Ndot_{V(f)},(t,\z)}^0(\0)$; we now have all our relevant data in terms of L\^{e} numbers and characteristic polar multiplicities of $\Ndot_{V(f)}$.   

\subsection{Main Result}\label{subsec:mainresult}
%\smallskip

The goal is then to decompose this data into numerical invariants which refer only to the $t=0$ and $t\neq 0$ slices of $V(f)$. So, in order to realize this goal, the next step is to decompose $\lambda_{\Ndot_{V(f)},(t,\z)}^0(\0)$ and $\lambda_{f,(t,\z)}^i(\0)$ for $i \geq 1$.
%\medskip

The 1-dimensional L\^{e} number $\lambda_{f,(t,\z)}^1(\0)$ is the easiest; by the dynamic intersection property for proper intersections,
\begin{align*}
\lambda_{f,(t,\z)}^1(\0) &= \left (\Lambda_{f,(t,\z)}^1 \cdot V(t) \right )_\0 \\
&= \sum_{p \in B_\epsilon \cap V(t-t_0)} \left ( \Lambda_{f,(t,\z)}^1 \cdot V(t-t_0) \right )_p \\
&= \sum_{p \in B_\epsilon \cap V(t-t_0)} \lambda_{f_{t_0},\z}^0(p).
\end{align*}

%\smallskip

The approach for $\lambda_{f_0,\z}^i(\0)$ for $i \geq 1$ is similar: we will use the fact that $f$ is an IPA-deformation of $f_0$ to ``move" around the origin in the $V(t)$ slice, and then use the dynamic intersection property. 

\begin{prop}\label{prop:lambda1}
If $(t,\z) = (t,z_1,\cdots,z_i)$ is an IPA-tuple for $f$ at $\0$ for $i \geq 1$, the following equality of intersection numbers holds:
\begin{align*}
\lambda_{f_0,\z}^i(\0)  = \sum_{q \in B_\epsilon \cap V(t-t_0,z_1,z_2,\cdots,z_i)} \lambda_{f_{t_0},\z}^i(q) 
\end{align*}
where  $0< |t_0| \ll \epsilon \ll 1$
\end{prop}

\begin{proof}
First, recall that $\lambda_{f_0,\z}^i(\0) = \left ( \Lambda_{f_0,\z}^i \cdot V(z_1,\cdots,z_i) \right )_\0$, where $\Lambda_{f_0,\z}^i$ is the $i$-dimensional L\^{e} cycle of $f_0$ with respect to $\z$ (see \secref{sec:classicalcycles}, as well as \cite{lecycles}).  For $i \geq 1$, we have 
\[
\Lambda_{f_0,\z}^i = \Lambda_{f,(t,\z)}^{i+1} \cdot V(t),
\]
so, by the dynamic intersection property,

\begin{align*}
\lambda_{f_0,\z}^i(\0) &=  \left ( \Lambda_{f,(t,\z)}^{i+1} \cdot V(t,z_1,\cdots,z_i) \right )_\0 \\
&=\sum_{q \in B_\epsilon \cap V(t-t_0,z_1,z_2,\cdots,z_i)} \left (\Lambda_{f,(t,\z)}^{i+1} \cdot V(t-t_0,z_1,z_2,\cdots,z_i) \right )_q \\
&= \sum_{q \in B_\epsilon \cap V(t-t_0,z_1,z_2,\cdots,z_i)} \left (\Lambda_{f_{t_0},\z}^{i} \cdot V(z_1,z_2,\cdots,z_i) \right )_q \\
&= \sum_{q \in B_\epsilon \cap V(t-t_0,z_1,z_2,\cdots,z_i)} \lambda_{f_{t_0},\z}^i(\0),
\end{align*}
where the second equality follows from the equality of cycles $\Lambda_{f,\z}^{i+1} \cdot V(t-t_0) = \Lambda_{f_{t_0},\z}^i$.
\end{proof}

%\smallskip

We can now state and prove our main result.

\begin{thm}\label{thm:main}
Suppose that $\pi : (\D \times \wt {V(f_0)}, \{0\} \times S) \to (V(f),\0)$ is a one-parameter unfolding with an isolated instability of a parameterized hypersurface $\im \pi_0 = V(f_0)$.  Suppose further that $\z = (z_1,\cdots,z_n)$ is chosen such that $\z$ is an IPA-tuple for $f_0 = f_{|_{V(t)}}$ at $\0$. Then, the following formulas hold for the L\^{e} numbers of $f_0$ with respect to $\z$ at $\0$: for $0 < |t_0|  \ll \epsilon \ll 1$,
\begin{align*}
\lambda_{f_0,\z}^0(\0) &= -\lambda_{\Ndot_{V(f_0)},\z}^0(\0) + \sum_{p \in B_\epsilon \cap V(t-t_0)} \left ( \lambda_{{f_{t_0}},\z}^0(p) + \lambda_{\Ndot_{V(f_{t_0})},\z}^0(p) \right ), 
\end{align*}
and, for $1 \leq i \leq n-2$,
\begin{align*}
\lambda_{f_0,\z}^i(\0) &= \sum_{q \in B_\epsilon \cap V(t-t_0,z_1,z_2,\cdots,z_i)} \lambda_{f_{t_0},\z}^{i}(q).
\end{align*}
%\smallskip

In particular, the following relationship holds for $0 \leq i \leq n-2$:
\begin{align*}
\lambda_{f_0,\z}^i(\0) +\lambda_{\Ndot_{V(f_0)},\z}^i(\0) =   \sum_{p \in B_\epsilon \cap V(t-t_0,z_1,z_2,\cdots,z_i)} \left ( \lambda_{{f_{t_0}},\z}^i(p) + \lambda_{\Ndot_{V(f_{t_0})},\z}^i(p) \right ) 
\end{align*}

\end{thm}

\begin{proof}
By Proposition~\ref{prop:basic} and Proposition~\ref{prop:lambda1}, it suffices to prove
\begin{align}\label{eqn:ndotdeformation}
\lambda_{\Ndot_{V(f)},(t,\z)}^0(\0) = -\lambda_{\Ndot_{V(f_0)},\z}^0(\0) + \sum_{p \in B_\epsilon \cap V(t-t_0)} \lambda_{\Ndot_{V(f_{t_0})},\z}^0(p).
\end{align}
Since $(t,\z)$ is an IPA-tuple for $f$ at $\0$, Theorem~\ref{thm:IPAimplies2} yields
\begin{align*}
\lambda_{\Ndot_{V(f_0)},\z}^0(\0) = \lambda_{\Ndot_{V(f)},(t,\z)}^1(\0) - \lambda_{\Ndot_{V(f)},(t,\z)}^0(\0),
\end{align*}
where $\Ndot_{V(f_0)} \cong {\Ndot_{V(f)}}_{|_{V(t)}}[-1]$ (cf. Remark~\ref{rem:ndotslicing}).

The main claim then follows by the dynamic intersection property for proper intersections applied to $\Lambda_{\Ndot_{V(f)},(t,\z)}^1$ (see Remark~\ref{rem:charpolcycle}):
\begin{align*}
\lambda_{\Ndot_{V(f)},(t,\z)}^1(\0) &= \left (\Lambda_{\Ndot_{V(f)},(t,\z)}^1 \cdot V(t) \right )_\0 \\ 
&= \sum_{p \in B_\epsilon \cap V(t-t_0)} \left (\Lambda_{\Ndot_{V(f)},(t,\z)}^1 \cdot V(t-t_0) \right )_p \\
&= \sum_{p \in B_\epsilon \cap V(t-t_0)} \lambda_{\Ndot_{V(f_{t_0})},\z}^0(p),
\end{align*}
for $0 < |t_0| \ll \epsilon \ll 1$.

%\medskip

Finally, we examine the relationship
\begin{align*}
\lambda_{f_0,\z}^i(\0) +\lambda_{\Ndot_{V(f_0)},\z}^i(\0) =   \sum_{p \in B_\epsilon \cap V(t-t_0,z_1,z_2,\cdots,z_i)} \left ( \lambda_{{f_{t_0}},\z}^i(p) + \lambda_{\Ndot_{V(f_{t_0})},\z}^i(p) \right ).
%\lambda_{g_0,\z}^i(\0) + \lambda_{\Ndot_{f_0},\z}^i(\0) &= \sum_{q \in B_\epsilon \cap V(t-t_0,z_1-\xi)} \left ( \lambda_{{g_{t_0}}_{|_{V(z_1-\xi)}},\hat{\z}}^{i-1}(q) +\lambda_{{\Ndot_{f_{t_0}}}_{|_{V(z_1-\xi)}},\hat{z}}^{i-1}(q) \right )
\end{align*}
For $i=0$, this follows by a trivial rearrangement of the terms in our expression for $\lambda_{f_0,\z}^0(\0)$. For $i \geq 1$, this is just Proposition~\ref{prop:lambda1} combined with Theorem~\ref{thm:IPAimplies2} and the dynamic intersection property on $\lambda_{\Ndot_{V(f)},(t,\z)}^i(\0)$, as in Proposition~\ref{prop:lambda1} for $\lambda_{f,(t,\z)}^i(\0)$.
\end{proof}

\begin{rem}
The relationship 
\begin{align*}
\lambda_{f_0,\z}^i(\0) +\lambda_{\Ndot_{V(f_0)},\z}^i(\0) =   \sum_{p \in B_\epsilon \cap V(t-t_0,z_1,z_2,\cdots,z_i)} \left ( \lambda_{{f_{t_0}},\z}^i(p) + \lambda_{\Ndot_{V(f_{t_0})},\z}^i(p) \right )
\end{align*}
suggests a sort of ``conserved quantity" between the sum of the L\^{e} numbers of $f_t$ and the characteristic polar multiplicities of $\Ndot_{V(f_t)}$ in one parameter deformations of parameterized hypersurfaces. It is a very interesting question to see how this relates to results in \cite{ndotMHMhep} regarding the structure of $\Ndot_{V(f)}$ as a mixed Hodge module, and the isomorphism $\Ndot_{V(f)} \cong \ker \{ \Id -\wt T_f\}$.
\end{rem}

\begin{exm}
We wish to examine Theorem~\ref{thm:main} in the context of Milnor's double point formula, where $\pi: (\D \times \C,\{0\} \times S) \to (\C^3,\0)$ parameterizes a deformation of the curve $V(f_0)$ into a curve $V(f_{t_0})$ with only double-point singularities (we can identify this deformed curve with the complex link $\mathbb{L}_{V(f),\0}$ inside the total deformation $V(f)$).  In this case, $\dim_\0 \Sigma f_0 = 0$, so the only non-zero L\^{e} number of $f_0$ is $\lambda_{f_0,z}^0(\0)$, where $z$ is any non-zero linear form on $\C^2$, and $\lambda_{f_0,z}^0(\0) = \mu_\0(f_0)$. 

%\smallskip

It is then an easy exercise to see that $\lambda_{\Ndot_{V(f_{t_0})},z}^0(p) = m(p) = |\pi^{-1}(p)| -1$ for $t_0$ small (and possibly zero) and $p \in \Sigma f$. 
%\smallskip

All together, this gives, by Theorem~\ref{thm:main}
\begin{align*}
\mu_\0(f_0) &= -(r-1) + \sum_{p \in B_\epsilon \cap V(t-t_0)} \left (\mu_p(f_{t_0}) + |\pi^{-1}(p)| -1 \right ) \\
&= 2\delta -r + 1,
\end{align*}
as there are $\delta$ double-points in the deformed curve $V(f_{t_0})$. We have thus recovered Milnor's original double-point formula for the Milnor number of a plane curve singularity. We picture this computation below.

\begin{center}
\includegraphics[width=9cm,height=6cm]{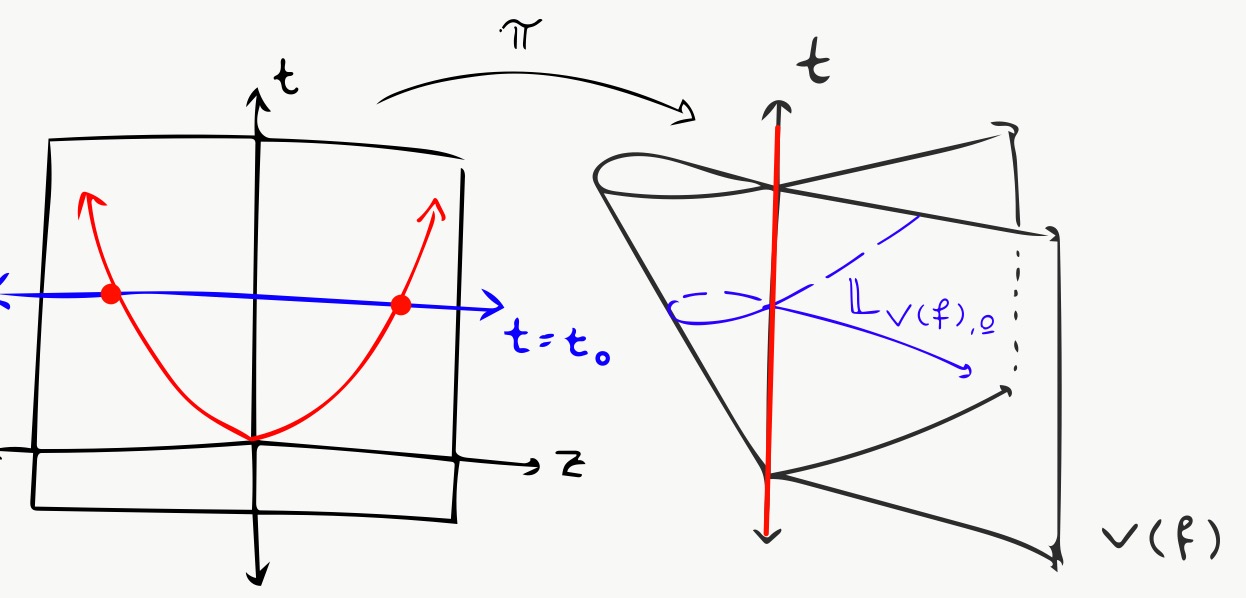}
\end{center}

\end{exm}

In analogy to plane curve singularities deforming into node singularities, it is well-known (see, e.g., \cite{surfacetoC3}), that for stabilizations of finitely determined maps $\pi_0 : (\C^2,S) \to (\C^3,\0)$, the image surface $\im \pi_0 = V(f_0)$ splits into cross caps (i.e., Whitney umbrellas), triple points, and $A_1$-singularities (i.e., nodes, which appear off the hypersurface on the relative polar curve).  These numbers are independent of the stabilization chosen, and depend only on $\pi_0$. We give the precise definition of finite determinacy of maps in \secref{sec:singmaps}.

Unfortunately, detecting these invariants using characteristic polar multiplicities and \thmref{thm:main} will have an unavoidable problem: we will also see points that belong to the \textbf{absolute polar curve $\Gamma_\z^1(\Sigma f)$}, which lie in the smooth part of $\Sigma f$ near $\0$, and are artifacts of our choice of linear forms $\z$ in calculating the characteristic polar multiplicities. For $\z = (t,z)$ a generic pair of linear forms on $\C^4$, the absolute polar curve of $\Sigma f$ at $\0$ is 
$$
\Gamma_\z^1(\Sigma f) = \overline{\Sigma \left ( (t,z)_{|_{\Sigma f}} \right)- \Sigma (\Sigma f)}
$$
(see \cite{leteissierpolar}, \cite{teissiervp2}, but we instead index by dimension instead of codimension).  Consequently, if $p \in \Gamma_\z^1(\Sigma f) \backslash \{\0\}$, we see that $\lambda_{\Ndot_{V(f_{t_0})},z}^0(p) \neq 0$ even if the stalk cohomology of $\Ndot_{V(f_{t_0})}$ is locally constant near $p$. This problem does not occur in the case of Milnor's original result, since the topology of the complex link of $\Sigma f$ at $\0$ is just a that of a finite set of points.

 We thus obtain the following result:

\begin{thm}\label{thm:surfacedeformation}
Suppose $\pi : (\D \times \C^2,\{\0\} \times S) \to (\C^4,\0)$ is a one-parameter unfolding of a finitely-determined map germ $\pi_0 :(\C^2,S) \to (\C^3,\0)$ parameterizing a surface $V(f_0) \subseteq \C^3$. Then, 
\begin{align*}
\lambda_{\Ndot_{V(f_0)},\z}^0(\0) = T +C - \delta + P
\end{align*}
where $T, C$, $\delta$, and $P$ denote the number of triple points, cross caps, $A_1$-singularities appearing in a stable deformation of $V(f_0)$, respectively, and if $V(f) = \im \pi$, $P$ denotes the number of intersection points of the absolute polar curve $\Gamma_{(t,z)}^1(\Sigma f)$ with a generic hyperplane $V(z)$ on $\C^4$ for which $(t,z)$ is an IPA-tuple for $f$ at $\0$.
\end{thm}
\begin{proof}
This follows directly from \thmref{thm:IPAimplies2}, \remref{rem:finitelydeterminedIPA}, \lemref{lem:IPAimplies1}, and recalling that $\lambda_{\Ndot_{V(f_{t_0})},\z}^0(\0) =1$ for both Whitney umbrellas and triple point singularities in $\C^3$ (see \exref{exm:ndotwhitneyumbrella} and \exref{exm:ndottriplepoint}). The $\delta$ term is equal to the degree of the relative polar curve $\Gamma_{f,t}^1$ at the origin, i.e., 
$$
\delta = \left ( \Gamma_{f,t}^1 \cdot V(t) \right )_\0 = \lambda_{\Ndot_{V(f),t}}^0(\0).
$$
\end{proof}

In fact, we can explicitly identify the Euler characteristic $\lambda_{\Ndot_{V(f_0)},\z}^0(\0)- \lambda_{\Ndot_{V(f_0)},\z}^1(\0)$ using \thmref{thm:surfacedeformation}.

\begin{cor}\label{cor:surfacedeformation}
Let $\pi_0$, $\pi$, $T,C,\delta,$ and $P$ be as in \thmref{thm:surfacedeformation}. Then, the following equalities hold:
\begin{align*}
\lambda_{\Ndot_{V(f_0)},\z}^0(\0) - \lambda_{\Ndot_{V(f_0)},\z}^1(\0) &= \chi(\Ndot_{V(f_0)})_\0 = -|\pi_0^{-1}(\0)| +1 = C-T-\delta - \chi(\mathbb{L}_{\Sigma f,\0}),
\end{align*}
where $\mathbb{L}_{\Sigma f,\0} \cong F_{t_{|_{\Sigma f}},\0}$ denotes the complex link of $\Sigma f$ at $\0$.
\end{cor}

\begin{rem}
Before we give the proof of \corref{cor:surfacedeformation} using derived category techniques, we will give a down-to-earth topological argument. The key idea in our proof is that one can compute the term $P$ using constant $\Z$ coefficients instead of $\Ndot_{V(f)}$, since $\Ndot_{V(f)}$ generically has stalk cohomology $\Z$ along $\Sigma f$ for hypersurfaces $V(f)$ that are the image of finitely-determined map germs \cite{surfacetoC3}.
\end{rem}

\begin{proof}(topological argument)
We compute the Euler characteristic of the pair $\chi(\mathbb{L}_{\Sigma f,\0},\mathbb{L}_{\Sigma f_0,\0})$. We can use $t$ to compute the complex link of $\Sigma f$ and $z$ to compute the complex link of $\Sigma f_0 = V(t) \cap \Sigma f$. This pair of subspaces makes sense, using the fact that $f$ is an IPA-deformation of $f_0$, and the complex link $\mathbb{L}_{\Sigma f_0,\0}$ of $\Sigma f_0$ is a finite set of points, and their multiplicity is unchanged as one moves in the $t$ direction away from the origin, pictured below:
\begin{center}
\includegraphics[width=4in,height=3in]{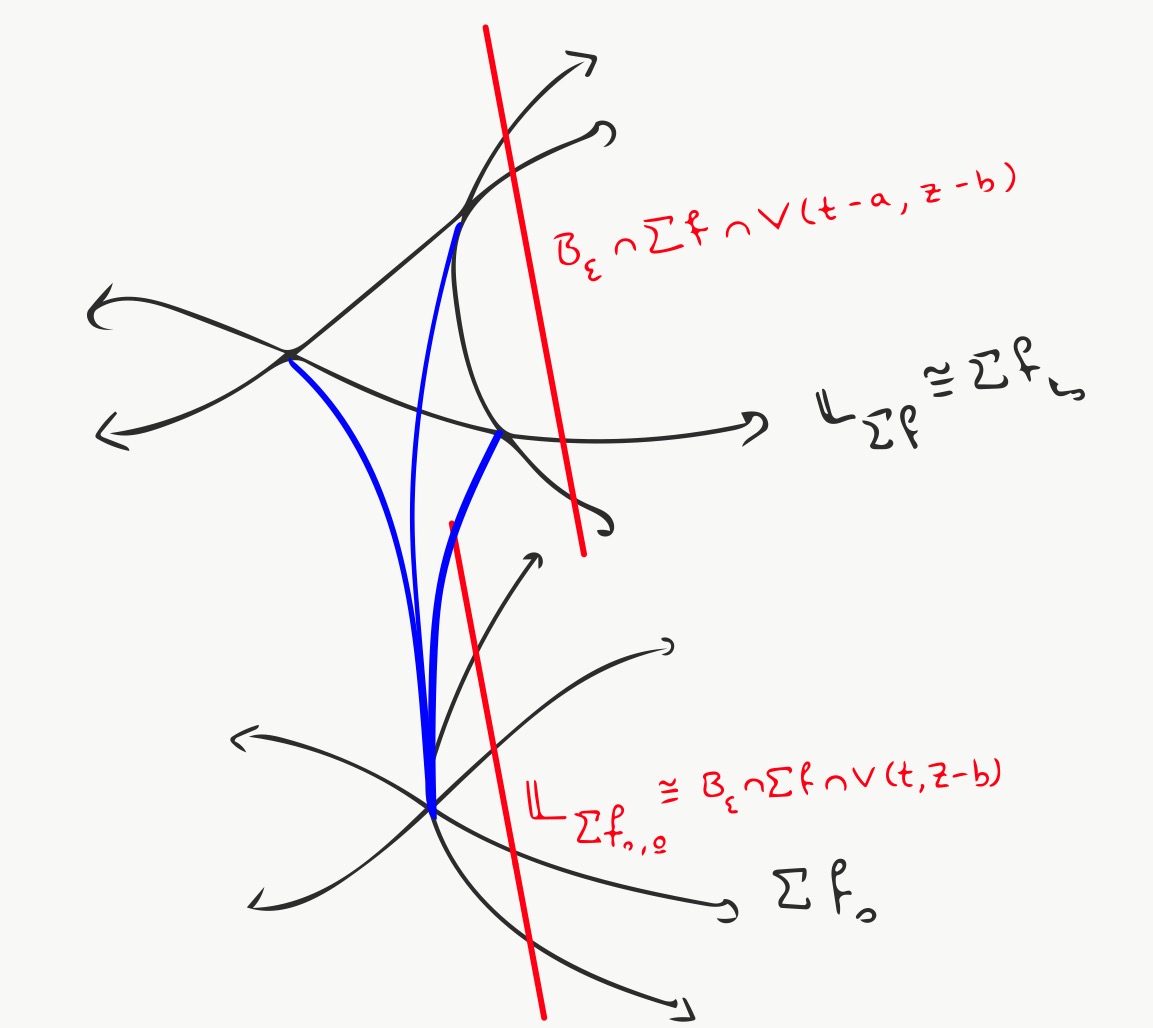}
\end{center}
Thus, we can identify $\mathbb{L}_{\Sigma f_0,\0} = B_\epsilon \cap \Sigma f \cap V(t,z-b)$ with $B_\epsilon \cap \Sigma f \cap V(t-a,z-b)$ for $0 < |a| \ll |b| \ll \epsilon \ll 1$. Consequently, we can identify 
$$
\chi(\mathbb{L}_{\Sigma f,\0},\mathbb{L}_{\Sigma f_0,\0}) = \chi(\phi_z[-1]\Z_{\mathbb{L}_{\Sigma f,\0}}^\bullet[1])_\0 = \sum_p \lambda_{\Z_{\Sigma f_{t_0}}^\bullet[1],z}^0(p).
$$
As the value of $z$ grows from $0$ to $b$, we pick up cohomological contributions (in the form of a non-zero multiplicity $\lambda_{\Z_{\Sigma f_{t_0}}^\bullet[1],z}^0(p)$) as we pass through points of the curves of triple points, cross caps, and the absolute polar curve with respect to $(t,z)$, pictured below:
\begin{center}
\includegraphics[width=5in,height=3in]{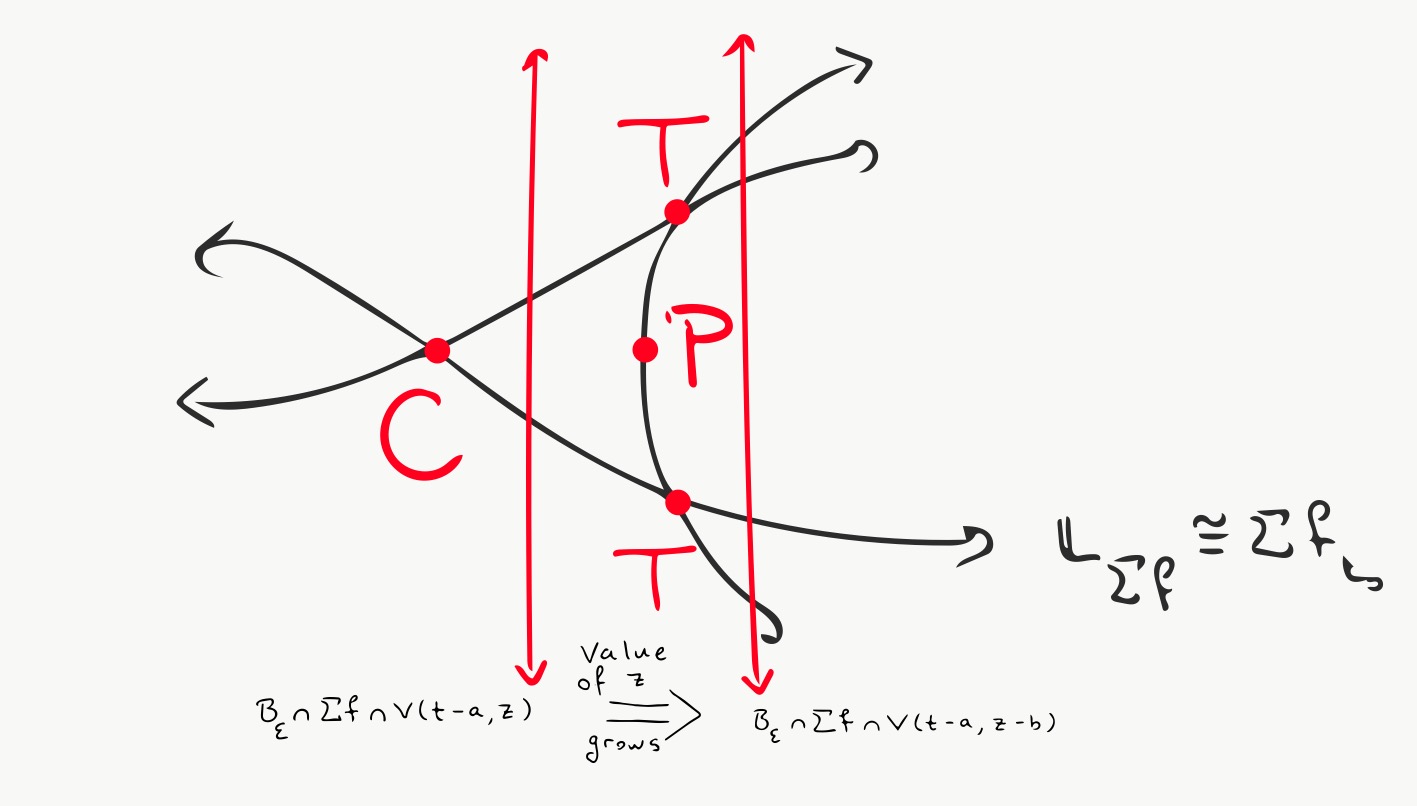}
\end{center}
At triple points, $\lambda_{\Z_{\Sigma f_{t_0}}^\bullet[1],z}^0(p) =2$, and at cross caps  $\lambda_{\Z_{\Sigma f_{t_0}}^\bullet[1],z}^0(p) = 0$. We count the contribution from the absolute polar curve as $P = \left (\Gamma_{(t,z)}^1(\Sigma f) \cdot V(z) \right )_\0$.
\begin{align*} 
2T+P = \chi(\mathbb{L}_{\Sigma f,\0},\mathbb{L}_{\Sigma f_0,\0}) &= \chi(\mathbb{L}_{\Sigma f,\0}) - \chi(\mathbb{L}_{\Sigma f_0,\0}) \\
&= \chi(\mathbb{L}_{\Sigma f,\0}) - \lambda_{\Ndot_{V(f_0)},z}^1(\0).
\end{align*}
Solving for $P$ and plugging the resulting expression into \thmref{thm:surfacedeformation} gives the result.
\end{proof}

\begin{proof}(perverse sheaves argument)
We wish to better understand the contribution of the term $P$ coming from the absolute polar curve of $\Sigma f$ appearing in \thmref{thm:surfacedeformation}. First, we note that these terms come from the 0-dimensional characteristic polar multiplicities $\lambda_{\Ndot_{V(f_{t_0})},\z}^0(p)$ in the expansion of $\lambda_{\Ndot_{V(f)},(t,z)}^1(\0)$, where $p$ is a smooth point of $\Sigma f$ in the $V(t-t_0)$ slice. Since the transverse singularity type of the image of a finitely-determined map is always that of a Morse function (see e.g., Mond \cite{surfacetoC3}), the stalk cohomology of $\Ndot_{V(f)}$ is $\Z$ at all smooth points of $\Sigma f$.  Consequently, we can calculate $P$ using the constant sheaf $\Z_{\Sigma f}^\bullet[2]$ in place of $\Ndot_{V(f)}$. 

However, $\Z_{\Sigma f}^\bullet[2]$ is not necessarily a perverse sheaf. To deal with this, note that, for all $t_0 \neq 0$, the restriction $\left (\Z_{\Sigma f}^\bullet[2] \right )_{|_{V(t-t_0)}} \cong \Z_{\Sigma f_{t_0}}^\bullet[1]$ is a perverse sheaf (the shifted constant sheaf on a curve is always perverse), and therefore $\psi_t[-1]\Z_{\Sigma f}^\bullet[2]$ is perverse. 

We then examine Euler characteristics at the origin of the distinguished triangle
\begin{equation}\label{eqn:disttri}
\left (\psi_t[-1]\Z_{\Sigma f}^\bullet[2] \right )_{|_{V(z)}}[-1] \to \psi_z[-1]\psi_t[-1]\Z_{\Sigma f}^\bullet[2] \to \phi_z[-1]\psi_t[-1]\Z_{\Sigma f}^\bullet[2] \xrightarrow{+1},
\end{equation}
where $\psi_z[-1]\psi_t[-1]\Z_{\Sigma f}^\bullet[2]$ and $\phi_z[-1]\psi_t[-1]\Z_{\Sigma f}^\bullet[2]$ are perverse sheaves for which $\0$ is an isolated point in their support. By \defref{def:charpol}, 
$$
\chi(\phi_z[-1]\psi_t[-1]\Z_{\Sigma f}^\bullet[2])_\0 = \rank H^0(\phi_z[-1]\psi_t[-1]\Z_{\Sigma f}^\bullet[2])_\0 = \lambda_{\Z_{\Sigma f}^\bullet[2],(t,z)}^1(\0).
$$

To calculate $\chi(\psi_z[-1]\psi_t[-1]\Z_{\Sigma f}^\bullet[2])_\0$, note that $\dim_\0 \supp \phi_t[-1]\Z_{\Sigma f}^\bullet[2] \leq 0$ (since $f$ is an IPA-deformation of $f_{|_{V(t)}}$ at $\0$) implies $\psi_z[-1]\phi_t[-1]\Z_{\Sigma f}^\bullet[2] = 0$, and so 
$$
\psi_z[-1]\Z_{\Sigma f_0}^\bullet[1] \xrightarrow{\thicksim} \psi_z[-1]\psi_t[-1]\Z_{\Sigma f}^\bullet[2].
$$
Thus, $\chi(\psi_z[-1]\psi_t[-1]\Z_{\Sigma f}^\bullet[2])_\0 = \chi(\psi_z[-1]\Z_{\Sigma f_0}^\bullet[1])_\0 = \lambda_{\Z_{\Sigma f_0}^\bullet[1],z}^1(\0)$. It is easy to see that $\lambda_{\Z_{\Sigma f_0}^\bullet[1],z}^1(\0) = \lambda_{\Ndot_{V(f_0)},z}^1(\0)$, since the transverse singularity type of $\Sigma f_0$ is that of a Morse function. 

Finally, we see that $\chi(\left (\psi_t[-1]\Z_{\Sigma f}^\bullet[2]\right )_{|_{V(z)}}[-1])_\0 = \chi(F_{t_{|_{\Sigma f}},\0}) = \chi(\mathbb{L}_{\Sigma f,\0})$, and we obtain the following formula from taking the Euler characteristic of (\ref{eqn:disttri}):
\begin{equation}\label{eqn:altsum}
\chi(\mathbb{L}_{\Sigma f,\0}) - \lambda_{\Ndot_{V(f_0)},z}^1(\0) + \lambda_{\Z_{\Sigma f}^\bullet[2],(t,z)}^1(\0) = 0.
\end{equation}
Using the dynamic intersection property, 
$$
\lambda_{\Z_{\Sigma f}^\bullet[2],(t,z)}^1(\0) = \sum_{p \in B_\epsilon \cap V(t-t_0)} \lambda_{\Z_{\Sigma f_{t_0}}^\bullet[1],z}^0(p) = 2T+P,
$$
since $\lambda_{\Z_{\Sigma f_{t_0}}^\bullet[1],z}^0(p) = 2$ when $p$ is a triple point singularity, and $\lambda_{\Z_{\Sigma f_{t_0}}^\bullet[1],z}^0(p) = 0$ when $p$ is a cross-cap singularity. The remaining terms, as in \thmref{thm:surfacedeformation}, come from the absolute polar curve of $\Sigma f$ with respect to $V(z)$. Consequently, we can solve for $P$ using (\ref{eqn:altsum})
$$
P = \lambda_{\Ndot_{V(f_0)},z}^1(\0) -\chi(\mathbb{L}_{\Sigma f,\0}) - 2T.
$$
Plugging this expression for $P$ into \thmref{thm:surfacedeformation} tells us
\begin{align*}
\lambda_{\Ndot_{V(f_0)},z}^0(\0) &= T+ C- \delta + P \\
&= T+ C- \delta + \lambda_{\Ndot_{V(f_0)},z}^1(\0) -\chi(\mathbb{L}_{\Sigma f,\0}) - 2T
\end{align*}
and so 
\begin{align*}
\chi(\Ndot_{V(f_0)})_\0 = \lambda_{\Ndot_{V(f_0)},z}^0(\0) - \lambda_{\Ndot_{V(f_0)},z}^1(\0) = C-T-\delta -\chi(\mathbb{L}_{\Sigma f,\0}).
\end{align*}
Finally, the Corollary follows from the fact that $\Ndot_{V(f_0)}$ has stalk cohomology concentrated in degree $-1$ (by \thmref{thm:qhomcriterion} and \remref{rem:charpolcycleeuler})
\end{proof}

\begin{rem}\label{rem:surfacedeformation}
If $V(f_0)$ is itself a $\Q$-homology manifold, then $\Ndot_{V(f_0)} = 0$. In this case, \thmref{thm:surfacedeformation} tells us that, in a stabilization $V(f)$ of $V(f_0)$, we have
$$
\chi(\mathbb{L}_{\Sigma f,\0}) = C-T-\delta.
$$
This scenario happens, for example, in L\^{e}'s Conjecture below.
\end{rem}

\subsection{Relationship with L\^{e}'s Conjecture}\label{subsec:leconj} Parameterized surfaces in $\C^3$ are the subject of a long-standing conjecture by L\^{e} D\~{u}ng Tr\'{a}ng \cite{1983Ntes}, in the vein of classical equisingularity problems of Mumford \cite{1961Pmdl} and Zariski, and is related to Bobadilla's Conjecture \cite{bobleconj}.
\begin{conj}[L\^{e}]
Suppose $(V(f),\0) \subseteq (\C^3,\0)$ is a reduced hypersurface with $\dim_\0 \Sigma f = 1$, for which the normalization of $V(f)$ is a bijection. Then, in fact, $V(f)$ is the total space of an equisingular deformation of plane curve singularities.
\end{conj}

We note that the assumption of the normalization of $V(f)$ being a bijection is equivalent to $\Ndot_{V(f)} = 0$, and that the conjecture is equivalent to the vanishing $\phi_L[-1]\Z_{V(f)}^\bullet[2] = 0$ for generic linear forms $L$ on $\C^3$. 

The first case to examine for this conjecture is when $\pi : (\C^2,\0) \to (\C^3,\0)$ is a corank 1 one map, which we may take to mean that $\pi$ is a one parameter unfolding with parameter $t$. This is the case proved in Bobadilla's reformulation of L\^{e}'s Conjecture in \cite{bobleconj}, in which $\Sigma f$ contains a smooth curve through the origin. Using \thmref{thm:main}, we can provide an alternative proof. 

This is actually the degenerate case mentioned in the introduction! Recall the non-stable deformation formula (\ref{eqn:generalmilnor}):
$$
\mu_\0(f_0) = -m(\0) + \sum_{p \in B_\epsilon \cap V(t-t_0)} \left (\mu_p(f_{t_0}) + m(p) \right ),
$$
where $m(p) = |\pi^{-1}(p)|-1$. Since $\pi$ is a bijection, we must have $m(p) = 0$ for all $p \in V(f)$. Hence, 
$$
\mu_\0(f_0) = \sum_{p \in B_\epsilon \cap V(t-t_0)} \mu_p(f_{t_0}),
$$
so the result follows from the non-splitting result of Gabrielov \cite{gabrielov}, Lazzeri \cite{lazzerimono}, and L\^{e} \cite{leacampo} (where this equality implies $\Sigma f$ is smooth at the origin, and thus defines a $\mu$-constant family of curves).

The difficult part for the general conjecture is reducing the above case, where one does not know if $\pi$ is an unfolding. Since $\Ndot_{V(f)} = 0$, it will be difficult to adapt the results obtained in this paper toward the conjecture.

\subsection{Other Generalizations in the Literature}\label{subsec:otherstuff}

\begin{rem}
In Lemma 2.2 of \cite{mondbentwires}, David Mond also obtains the result that, for a stabilization of a plane curve singularity $V(f_0)$, one has 
$$
\mu_\0 \left ( t_{|_{V(f)}} \right ) = \delta - r +1,
$$
where $\mu_\0 \left ( t_{|_{V(f)}} \right )$ is called the \textbf{image Milnor number} of the stabilization.  It is an interesting question in general how one can relate the theory of map germs from $\C^n$ to $\C^{n+1}$ of finite $\mathcal{A}$-codimension (in Mather's nice dimensions ($n<15$) and beyond) to our result \thmref{thm:main}. Even more so would be to understand the relationship between this theorem and Mond's conjecture, since IPA-deformations are more general than deformations typically used in the singularity theory of maps.
\end{rem}

\begin{rem}
Gaffney also generalizes the result $\mu_\0 (t_{|_{V(f)}}) = \delta - r +1$ in \cite{l0equivalence}, although to the very different setting of maps $G: (\C^n,S) \to (\C^{2n},\0)$. In Theorem 3.2 and Corollary 3.3 of \cite{l0equivalence}, this formula is derived in terms of the Segre number of dimension 0 of an ideal associated to the image multiple-point set and the number of Whitney umbrellas of the composition of the map $G$ with a generic projection to $\C^{2n-1}$. 
\end{rem}

\begin{rem}\label{rem:Gaffneyrem}
In the case of finitely-determined maps $F : (\C^2,\0) \to (\C^3,\0)$ of the form $F(t,z^2,F_3(t,z))$, $\im F = V(f)$ defines a surface whose singular locus $\Sigma f$ is an \textbf{isolated complete intersection singularity} by results of Mond and Pellikaan (e.g., Prop. 2.2.4 of \cite{singmaps}). In this case, the results of \cite{gaffneypairs} apply, and we can recover Gaffney's formula (Proposition 2.4) for the $0$-dimensional L\^{e} number of $f$ at $\0$ 
$$
\lambda_{f,\z}^0(\0) = \delta + 2C+e(JM(\Sigma f)).
$$
where $\delta$ (resp., $C$) is the number of $A_1$-singularities (resp., cross caps) appearing in a stabilization of $F$, and $e(JM(\Sigma f))$ is the Buchbaum-Rim multiplicity of the Jacobian Module of $\Sigma f = D$. This follows directly from \thmref{thm:main}, using the fact that $e(JM(\Sigma f))$ gives the number of critical points of a generic linear form on the curves of multiple-points in the stabilization (which comprise the term $P$ used in \thmref{thm:surfacedeformation}). Finally, since $\Sigma f$ is an isolated complete intersection singularity, there can be no triple points in a stabilization of $F$. We would like to express our thanks to Terence Gaffney for pointing out this relationship.
\end{rem}

%One easily calculates the L\^{e} numbers of a Whitney umbrella (e.g., $f=y^2-x^3-tx^2$) to be $\lambda_{f_0,\z}^0(\0) = \lambda_{f_0,\z}^1(\0) = 1$, and for a triple point singularity (e.g., $f = xyz$) to be $\lambda_{f_0,\z}^0 = 2$ and $\lambda_{f_0,\z}^1(\0) = 3$. This, together with \thmref{thm:main} and \thmref{thm:surfacedeformation} yields the following result. 

%\begin{cor}\label{cor:milnortypesurface}
%Suppose that $\pi :( \D \times \C^2,\{0\} \times S) \to (\C^4,\0)$ is a one-parameter unfolding of a finitely determined map germ $\pi_0 : (\C^2,S) \to (\C^3,\0)$ parameterizing a surface $V(f_0) \subseteq \C^3$. Then, the L\^{e} numbers of $f_0$ are given by 
%\begin{align*}
%\lambda_{f_0,\z}^0(\0) &= -\delta+ 3T+2C \\
%\lambda_{f_0,\z}^1(\0) &= 3T + C.
%\end{align*}
%\end{cor}

%The type of result we obtain from \corref{cor:surfacedeformation} and in reproving Milnor's theorem suggests a more general ``formula". If we know all the ($\mathcal{A}$-)stable types of singularities that exist for maps $(\C^n,S) \to (\C^{n+1},\0)$ in Mather's nice dimensions $n < 15$, we can obtain a formula for the L\^{e} numbers of the image hypersurface $V(f_0)$ of any finitely determined map germ $\pi_0 : (\C^n,S) \to (\C^{n+1},\0)$ parameterizing $V(f_0)$. 

%{\color{red} {\Large DO THE CASE WHERE $V(f_0) \subseteq \C^4$ IS PARAMETERIZED.}}

It is a very interesting question to see what formulas might arise from \thmref{thm:main} when one works outside of Mather's nice dimensions; for $n \geq 15$, one can no longer approximate a finitely determined map with stable maps, but the relationship in \thmref{thm:main} still holds. 

\bigskip

\section{Appendix: The L\^{e} Cycles and Relative Polar Varieties}\label{sec:classicalcycles}

The L\^{e} numbers of a function with a non-isolated critical locus are the fundamental invariants we consider in this paper. First defined by Massey in \cite{levar1} and \cite{levar2}, these numbers generalize the Milnor number of a function with an isolated critical point. 

%\smallskip

The L\^{e} cycles and numbers of $g$ are classically defined with respect to a \textbf{prepolar-tuple} of linear forms $\z = (z_0,\cdots,z_n)$; loosely, these are linear forms that transversely intersect all strata of a good stratification of $V(g)$ near $\0$ (see, for example, Definition 1.26  of \cite{lecycles}).  The purpose of Proposition~\ref{prop:basic} in Section~\ref{sec:IPA} is to replace the assumption of prepolar-tuples with IPA tuples.
%\medskip

\begin{defin}\label{def:relpolclassical}
The \textbf{$k$-dimensional relative polar variety of $g$ with respect to $\z$}, at the origin, denoted $\Gamma_{g,\z}^k$, consists of those components of the analytic cycle $V \left (\frac{\pd g}{\pd z_k},\cdots,\frac{\pd g}{\pd z_n} \right )$ at the origin which \emph{are not} contained in $\Sigma g$.  
\end{defin}

\begin{defin}\label{def:lecycles}

The \textbf{$k$-dimensional L\^e cycle of $g$ with respect to $\z$}, at the origin, denoted $\Lambda_{g,\z}^k$, consists of those components of the analytic cycle $\Gamma_{g,\z}^{k+1} \cdot V \left ( \frac{\pd g}{\pd z_k} \right )$ which \emph{are} contained in $\Sigma g$.
\end{defin}

%\medskip

\begin{defin}\label{def:lenums}
The \textbf{$k$-dimensional L\^{e} number of $g$ at $p=(p_0,\cdots,p_n)$ with respect to $\z$}, denoted $\lambda_{g,\z}^k(p)$, is equal to the intersection number
\[
\left ( \Lambda_{g,\z}^k \cdot V(z_0-p_0,\cdots,z_{k-1}-p_{k-1}) \right )_p,
\]
provided this intersection is purely zero-dimensional at $p$.
\end{defin}

%\medskip

\begin{exm}
When $g$ has an isolated critical point at the origin, the only non-zero L\^{e} number of $g$ is $\lambda_{g,\z}^0(\0)$. In this case, we have:
\begin{align*}
\lambda_{g,\z}^0(\0) &= \left (\Lambda_{g,\z}^0 \cdot \U \right )_\0 \\
&= V\left (\frac{\pd g}{\pd z_0},\cdots,\frac{\pd g}{\pd z_n} \right )_\0,
\end{align*}
i.e., the 0-dimensonal L\^{e} number of $g$ is just the multiplicity of the Jacobian scheme. In the case of an isolated critical point, this is the Milnor number of $g$ at $\0$. 
\end{exm}

%\medskip

\begin{exm}
Suppose now that $\dim_\0 \Sigma g = 1$. Then, the only non-zero L\^{e} numbers of $g$ are $\lambda_{g,\z}^0(\0)$ and $\lambda_{g,\z}^1(p)$ for $p \in \Sigma g.$ 

At $\0$, we have 
\begin{align*}
\lambda_{g,\z}^1(\0) &= \left (\Lambda_{g,\z}^1 \cdot V(z_0) \right)_\0 \\
&= \left (V\left (\frac{\pd g}{\pd z_1},\cdots,\frac{\pd g}{\pd z_n} \right) \cdot V(z_0) \right)_\0 \\
&= 	\sum_{q \in B_\epsilon \cap V(z_0-q_0)\cap \Sigma g} \left (V\left (\frac{\pd g}{\pd z_1},\cdots,\frac{\pd g}{\pd z_n} \right) \cdot V(z_0-q_0) \right )_q \\
&= \sum_{q \in B_\epsilon \cap V(z_0-q_0)\cap \Sigma g} \mu_q \left (g_{|_{V(z_0-q_0)}} \right) 
\end{align*}
where the second to last line is the dynamic intersection property for proper intersections. 

%\smallskip

After rearranging the terms in the last line, we find
\begin{align*}
\lambda_{g,\z}^1(\0) = \sum_{C \subseteq \Sigma g \text{  irred. comp.}} \overset{\circ}{\mu}_C\left (C \cdot V(z_0) \right )_\0,
\end{align*}
where the sum is indexed over the collection of irreducible components of $\Sigma g$ at the origin, and $\overset{\circ}{\mu}_C$ denotes the generic transversal Milnor number of $g$ along $C$.
\end{exm}

\bigskip

\section{Appendix: Singularities of Maps}\label{sec:singmaps}
Our primary references for this appendix are \cite{mather}, \cite{GAFFNEY1993185}, and \cite{PMIHES_1968__35__127_0}.

Let $f : (\C^n,S) \to (\C^p,\0)$ be a holomorphic map (multi-)germ, with $S$ a finite subset of $\C^n$.  Then, the group of biholomorphisms $\Diff(N,S)$ from $\C^n$ to $\C^n$ (preserving $S$), acts on $f$ on the left by pre-composition; similarly, the group of biholomorphisms $\Diff(\C^p,\0)$ from $\C^p$ to to $\C^p$ (preserving the origin), acts on $f$ on the right by composition (we realize the notation $\Diff$ to denote biholomorphisms seems confusing, but this appears to be standard notation). Thus, we have a group action of $\mathcal{A} := \Diff(\C^n,S) \times \Diff(\C^p,\0)$ on the space of all holomorphic maps $\mathcal{O}(n,p)$ from $(\C^n,S)$ to $(\C^p,\0)$:
\begin{align*}
\mathcal{A} \times \mathcal{O}(n,p) &\to \mathcal{O}(n,p) \\
(\Phi,\Psi)*f &= \Phi \circ f \circ \Psi^{-1}.
\end{align*}
Clearly, this group action defines an equivalence relation on $\mathcal{O}(n,p)$, where $f \thicksim g$ if there exists $(\Phi,\Psi) \in \mathcal{A}$ for which $\Phi^{-1} \circ f \circ \Psi = g$. Let $\mathcal{A}_e$ denote the pseudo-group gotten by allowing non-origin preserving equivalences, and $\mathcal{O}_e(n,p)$ the space of map-germs at the origin, but not necessarily origin-preserving.

\begin{defin}\label{def:generalunfolding}
A \textbf{$d$-parameter unfolding} of $f$ is a map germ 
$$
F : (\C^d \times \C^n,\{\0\} \times S) \to (\C^d \times \C^p,\0)
$$
of the form 
$$
F(\t,\z) = (\t,\wt f(\t,\z)),
$$
such that $\wt f(\0,\z) = f(\z)$, and $\t = (t_1,\cdots,t_d)$ are coordinates on $\C^d$. We also write $f_\t(\z) := \wt f(\t,\z)$, so that $f_0 = f$.

\smallskip

We say $F$ is a \textbf{trivial unfolding} of $f$ if there are $d$-parameter unfoldings of the identity on $\C^n$ and $\C^p$, say $\Phi$ and $\Psi$, respectively, such that $\Phi \circ F \circ \Psi^{-1} = (id,f)$. 
\end{defin}

\begin{defin}\label{def:stablemap}
We say $f \in \mathcal{O}_e(n,p)$ is \textbf{stable} if every unfolding of $f$ is trivial.
\end{defin}

\begin{defin}\label{def:stableunfolding}
We say an unfolding $F : (\C^d \times \C^n,\{0\} \times S) \to (\C^d \times \C^p,\0)$, $F(\t,\z) = (\t,f_\t(\z))$ of $f$ is a \textbf{stable unfolding} (or, a \textbf{stabilization}) of $f$ if $f_\t$ is stable for all $t \neq 0$.
\end{defin}

\begin{defin}\label{def:finitelydetermined}
We say that a map $f \in \mathcal{O}(n,p)$ is \textbf{finitely determined} if there exists an integer $k$ such that any $g \in \mathcal{O}(n,p)$ which has the same $k$-jet as $f$ satisfies $f \thicksim g$. That is, if, for all $x \in S$, the derivatives of $f$ and $g$ at $x$ of order $\leq k$ are the same (with respect to a system of coordinates at $x$ and $y$).
\end{defin}

We primarily care about (one-parameter) stabilizations of finitely-determined map germs for the fact that these maps all have isolated instabilities at the origin (see \secref{sec:unfold}). In general, we have the following remark.

\begin{rem}
Suppose, that $F$ is a stable one-parameter unfolding of a finite map $f$, and that $h :(\im F,\0) \to (\C,0)$ is the projection onto the unfolding parameter. Then a point $x \in V(h)$ is a point in the image of $f$. If $f$ is stable at $x$, then $h$ is locally a topologically trivial fibration in a neighborhood of $x$; consequently, the Milnor fiber is contractible, and $x \notin \Sigma_{\Top} h$.
Thus, $\Sigma_{\Top} h$ is contained in the unstable locus of $F_0$. We will need this observation in \secref{sec:unfold}.
\end{rem}

\printbibliography
\end{document}